\documentclass[11pt]{article}

\usepackage[T1]{fontenc}
\usepackage[utf8]{inputenc}
\usepackage[british]{babel}
\usepackage[
rm={oldstyle=false,proportional=true},
sf={oldstyle=false,proportional=true},
tt={oldstyle=false,proportional=true}
]{cfr-lm}
\usepackage[babel]{microtype}

\usepackage{authblk}

\usepackage[x11names, svgnames, rgb]{xcolor}
\usepackage{tikz}
\usetikzlibrary{decorations,arrows,shapes}

\usepackage{amsmath}
\usepackage{amsthm}
\usepackage{amssymb}
\usepackage{mathtools}
\usepackage{mathrsfs}

\usepackage{ragged2e}

\usepackage{multicol}
\usepackage{enumitem}

\usepackage[noadjust]{cite}
\usepackage[hidelinks,bookmarksnumbered]{hyperref}
\usepackage[capitalise]{cleveref}

\usepackage{orcidlink}

\newtheorem{theorem}{Theorem}[section]
\newtheorem{lemma}[theorem]{Lemma}
\newtheorem{proposition}[theorem]{Proposition}
\newtheorem{corollary}[theorem]{Corollary}

\newcommand{\thistheoremname}{}
\newtheorem*{genericthm*}{\thistheoremname}
\newenvironment{namedthm*}[1]
{\renewcommand{\thistheoremname}{#1}%
\begin{genericthm*}}
{\end{genericthm*}}

\theoremstyle{definition}
\newtheorem{definition}[theorem]{Definition}
\newtheorem{question}[theorem]{Question}
\newtheorem{problem}[theorem]{Open problem}

\DeclareMathOperator{\race}{r}
\DeclareMathOperator{\birth}{b}
\DeclareMathOperator{\outcomeR}{o_R}
\DeclareMathOperator{\flex}{f}
\DeclareMathOperator{\terminal}{T}

\newcommand{\longest}[1]{\overline{\mathsf L(#1)}}

\begin{document}

\title{On factorisations of Left dead ends}
\author{Alfie Davies\,\orcidlink{0000-0002-4215-7343}}
\affil{\small{Department of Mathematics and Statistics,\\
        Memorial University of Newfoundland,
        Canada
}}
\affil{\small{\emph{\href{mailto:research@alfied.xyz}{research@alfied.xyz}}}}
\date{}
\maketitle

\begin{abstract}
    Losing a game is difficult. Recent work of Larsson, Nowakowski, and Santos,
    as well as that of Siegel, has opened up enticing lines of research into
    partizan mis\`ere theory. The algebraic structure of mis\`ere monoids is
    not yet well understood. A natural object to study is a \emph{universe} of
    games---a set of games satisfying various closure properties. We prove that
    a universe is determined by its Left ends, thus motivating a study of such
    games, which was also noted by Siegel for dead ending universes in
    particular. We initiate this study by investigating the atomic structure of
    the monoid of Left dead ends, showing that it is reduced, pocancellative,
    and an FF-monoid. We give sharp bounds on the lengths of factorisations,
    and also build tools to find atoms. Some results require the introduction
    of a new concept called \emph{flexibility}, which is a measure of the
    distance from a game to integers. We also discover large families of games
    admitting unique factorisations; this includes the uncovering of a prime
    element: $\overline{1}$.
\end{abstract}

\section{Introduction}

Trying to lose a game is often harder than winning. That is, mis\`ere play has
proved itself an unruly twin to normal play. This reality has been a plight of
Combinatorial Game Theory for some decades. It is well-known that a shift began
with Plambeck and Siegel's work \cite{plambeck:taming, plambeck.siegel:misere},
which threw light upon the dark corners of impartial mis\`ere theory. More
recently, the results of Larsson, Nowakowski, and Santos
\cite{larsson.nowakowski.ea:absolute} have provided many lines of attack onto
the partizan theory. Their introduction of a \emph{universe}---a set of games
satifying various closure properties that we will explain in the next section
(see \cref{def:universe})---and a wonderful theorem on comparability when
restricted to such a structure, yield an effective tool that we may use to
better understand mis\`ere play.

In normal play, where the last player to move wins, the set of game
values\footnote{By the \emph{value} of a game, we mean its equivalence class
modulo `$=$'; cf.\ the \emph{game value} as defined by Siegel in
\cite[p.~12]{siegel:combinatorial}.} forms a group, and there are lots of games
that are comparable with one another. In mis\`ere play, however, where the last
player to move loses, not only does the set of game values not form a group,
but in fact there are no non-zero invertible elements at all, as shown first by
Mesdal and Ottoway \cite[Theorem 7 on p.~5]{mesdal.ottaway:simplification}.
Making things worse, the structure is rife with indistinguishable pairs of
games; indeed, from the 256 strict forms born by day 2, there are 256
equivalence classes in mis\`ere \cite[Theorem 6.2 on p.~233]{siegel:misere},
compared with 22 in normal play (see, for example,
\cite[p.~27]{calistrate.paulhus.ea:on}). There simply isn't a lot to work with;
at least, not as much as one would hope for, given the agreeability of normal
play.

The wonder of Larsson et al.'s work then is to create more workable mis\`ere
structures; they present a theorem \cite[Theorem 4 on
p.~6]{larsson.nowakowski.ea:absolute}, analogising the situation in normal
play, that allows us to determine the comparability of two games modulo a
universe. This is made all the more powerful by the fact that universes are
natural structures to define, and we already have some to play with (see
\cite{larsson.milley.ea:recursive, milley:restricted, milley.renault:dead,
allen:investigation}): the dicot universe $\mathcal{D}$, which is contained in
every other universe; the dead ending universe $\mathcal{E}$; and the full
mis\`ere universe $\mathcal{M}$. And the two universes $\mathcal{D}$ and
$\mathcal{E}$ do both have non-zero invertible elements. In fact, their
invertible elements have been characterised: Milley and Renault resolved the
dead-ending case \cite[Theorems 19 and 22 on
pp.~11--12]{milley.renault:invertible}; and Fisher, Nowakowski, and Santos
resolved the dicot case \cite[Theorem 12 on
p.~7]{fisher.nowakowski.ea:invertible}.

As fortune would have it, there are more universes than just these three.
Indeed, Siegel \cite[pp.~10--11 in \S3]{siegel:on} showed that there are
uncountably many universes lying between $\mathcal{D}$ and $\mathcal{E}$: the
\emph{dead-ending universes} (see \cite[p.~2]{siegel:on}), not to be confused
with \emph{the} dead-ending universe $\mathcal{E}$. This built off of other
work by Larsson et al.\ that showed that the set of such universes is at least
countably infinite \cite[Theorem 26 on
pp.~13--14]{larsson.nowakowski.ea:infinitely}---they also showed that there are
only two universes under normal play \cite[Theorem 14 on
pp.~7--8]{larsson.nowakowski.ea:infinitely}. Our first contribution is to show
that universes are entirely characterised by their Left ends. This result was
also discovered by Siegel \cite[Propositions 2.5 and 2.6 on p.~6]{siegel:on} in
the case of dead-ending universes.

\begin{namedthm*}{\cref{prop:left-ends}}
    If $\mathcal{A}$ is a set of games that is conjugate closed and hereditary,
    with $E_L(\mathcal{A})$ its set of Left ends, then
    $\mathcal{D}(\mathcal{A})=\mathcal{D}(E_L(\mathcal{A}))$.
\end{namedthm*}

An immediate consequence of this is that there exists a bijective
correspondence between universes and sets of Left ends that are additively
closed and hereditary (with respect to Left ends; i.e.\ for all Left ends in
the set, all subpositions that are Left ends must also be in the set). Thus, to
understand universes, it is paramount that we understand Left ends. But this is
not an easy problem. Consider, for example, any game $G\in\mathcal{M}$. Then
$\{\cdot\mid G\}$ is a Left end. And, to study a game form, one usually needs
to also study its subpositions, so understanding Left ends appears no different
than understanding general mis\`ere forms. What are we to do? Luckily for us,
there are natural families of Left ends that lend themselves much more
willingly to analysis: one such set is the family of \emph{Left dead ends},
whose subpositions are all Left ends.

Crucially, the options of a Left dead end are necessarily themselves Left dead
ends; we can no longer have arbitrary forms as options, as in our earlier
example of $\{\cdot\mid G\}$. By our correspondence between sets of Left ends
and universes, it is clear that (additively closed and hereditary) sets of Left
dead ends correspond precisely with dead-ending universes. Our lofty goal in
this paper then is to understand the structure of these Left dead ends, which
we begin to in the proceeding sections, so that we may better understand these
(dead-ending) universes. Siegel \cite[Theorem 3.1 on p.~7]{siegel:on} has
observed a simplified comparison test for Left dead ends (see
\cref{thm:reduced}) that we take full advantage of in this paper.

Our study has an algebraic slant; when equipped with the usual disjunctive
compound, the set of Left dead ends forms a commutative pomonoid.
Factorisations in monoids are well-studied, and there are lots of natural
questions to ask: when is an element an \emph{atom} (sometimes also called an
\emph{irreducible}, which is an element that cannot be written as a sum of two
non-unit elements); when do elements have factorisations (i.e.\ when can
elements be written as sums of atoms); how long can these factorisations be;
when are they unique? We investigate all of these questions as they pertain to
the monoid of Left dead ends.

First, in \cref{sec:prelims}, we give some technical background for both
Combinatorial Game Theory and Algebra.

In \cref{sec:properties}, we start to develop some tools to help us answer some
of the natural questions we might ask about our monoid. One of the main results
of the section is the following.

\begin{namedthm*}{\cref{thm:ff}}
    The monoid of Left dead ends $\mathcal{L}$ is an FF-monoid.
\end{namedthm*}

This says that, if $G$ is a Left dead end, then: $G$ has an factorisation;
every factorisation of $G$ is of bounded length; and there are only finitely
many factorisations of $G$.

In \cref{sec:prime}, we discover a prime element in our monoid of Left dead
ends: $\overline{1}=\{\cdot\mid0\}$.

\begin{namedthm*}{\cref{thm:one-prime}}
    The game $\overline{1}$ is prime in the monoid of Left dead ends
    $\mathcal{L}$.
\end{namedthm*}

For us, this means that, if $\overline{1}$ appears in some factorisation of a
Left dead end, then it appears in \emph{all} factorisations of that Left dead
end. We are not able to determine here whether any other atoms are prime; the
proof of this particular theorem relies heavily on a behaviour unique to
$\overline{1}$ (see \cref{prop:one-option,prop:one-good-option}).

In \cref{sec:flexibility}, we introduce the concept of \emph{flexibility},
which is a measure of the distance from a game $G$ to integers (see
\cref{def:flex}), and is written $\flex(G)$. We use this to idea to prove a
myriad of results. One such corollary is the following, where a \emph{good}
option is simply an option that is not strictly dominated by any other option
(see \cref{def:good-option}).

\begin{namedthm*}{\cref{cor:big-birth-atom}}
    If $G$ is a Left dead end with more than one good option, and
    $\birth(G)>2\cdot\flex(G)-2$, then $G$ is an atom.
\end{namedthm*}

This tells us that if the birthday of a Left dead end with more than one good
option gets \emph{too large} relative to its flexibility, then it must be an
atom---it cannot be written as a sum of two non-zero Left dead-ends.

Knowing that the factorisations of a game are of bounded length, it is natural
to then ask what bounds, if any, we can put on these lengths. We provide an
answer to this question in \cref{sec:factorisation-lengths}, where we write:
$\longest{G}$ for the length of a longest factorisation of $G$; $\terminal(G)$
for the set of lengths of maximal runs of $G$ (see
\cref{def:terminal-lengths}); and $\race(G)$ for the minimum element of
$\terminal(G)$ (see \cref{def:race}). Note that the inequality we provide is
indeed sharp.

\begin{namedthm*}{\cref{thm:one-option-bound,thm:length-bound}}
    If $G$ is a Left dead end with exactly one good option $G'$, then
    $\longest{G}=1+\longest{G'}$. Otherwise, if $G$ has more than one good
    option, then
    \[
        \longest{G}\leq\min\left\{\race(G),\longest{G'}+1,|\terminal(G)|-1,\frac{\flex(G)+1}{2}\right\},
    \]
    where $G'$ ranges over all good options of $G$.
\end{namedthm*}

This result allows us to recursively compute bounds for the lengths of
factorisations of any Left dead end $G$, which we have implemented in
\texttt{gemau} \cite{davies:gemau}---a basic software package we introduce to
aid in computations necessary in Combinatorial Game Theory, with a view to
provide good support for partizan mis\`ere theory specifically. Interested
readers should also look into \texttt{CGSuite} \cite{siegel:cgsuite}, written
by Aaron Siegel, which is the go-to system for these kinds of things.

It is also impossible to stop oneself from asking when a factorisation of a
game is unique. Could it be the case that every Left dead end is uniquely
factorisable? We are unable to find the answer to this question, but we make
significant progress. One simple result from the end of \cref{sec:uniqueness}
is the following.

\begin{namedthm*}{\cref{thm:strong}}
    If $G$ is a Left dead end with a good prime-factorisable option, then
    $n\cdot G$ is uniquely factorisable for all $n$.
\end{namedthm*}

\cref{sec:final-remarks} contains our final remarks where we discuss some open
problems and further directions for research.

\section{Preliminaries}
\label{sec:prelims}

As dead ends will be the spotlight of our investigation, we recall some
standard definitions from Milley and Renault \cite{milley:restricted,
milley.renault:dead}. We will not, however, recall the basics of the field of
Combinatorial Game Theory, and the reader is directed to Siegel
\cite{siegel:combinatorial} for any information pertaining to that. Firstly, we
emphasise that we shall be considering only \emph{short} games, although we
will occasionally still use the word to reinforce this.

So, as in \cite[Definition 2.1.8 on p.~10]{milley:restricted}, we will refer to
a game with no Left option as a \emph{Left end}; with Right ends defined
analogously. Then, as in \cite[Definition 1]{milley.renault:dead} and
\cite[Definition 5.1.1 on p.~75]{milley:restricted}, a Left end is called a
\emph{Left dead end}\footnote{Some authors may write `dead Left end,' or other
similar phrases, for `Left dead end,' as in \cite{milley.renault:restricted}.}
if every subposition is also a Left end; again, Right dead ends are defined
analogously. A game is \emph{dead ending} if every subposition that is an end
is a dead end.

As we have discussed in the introduction, this work is motivated by trying to
understand universes of games, which, following Siegel, we define thusly.

\begin{definition}[{\cite[Definition 1.4 on p.~2]{siegel:on}}; cf.\
    {\cite[Definition 12 on pp.~17--18]{larsson.nowakowski.ea:absolute}}]
    \label{def:universe}
    A set of games $\mathcal{U}$ is a \emph{universe}\footnote{
        Some authors would relax the definition of universe to not include the
        dicotic closure property (which has also been referred to as `parental
        closure'), as in \cite{larsson.nowakowski.ea:absolute}.
    }
    if the following closure properties are satisfied for all
    $G,H\in\mathcal{U}$:
    \begin{enumerate}
        \item
            (additive closure) $G+H\in\mathcal{U}$;
        \item
            (hereditary closure) $G'\in\mathcal{U}$ for every option $G'$ of
            $G$;
        \item
            (conjugate closure) $\overline{G}\in\mathcal{U}$; and
        \item
            (dicotic closure) if $\mathcal{G}$ and $\mathcal{H}$ are non-empty,
            finite subsets of $\mathcal{U}$, then
            \[
                \{\mathcal{G}\mid\mathcal{H}\}\in\mathcal{U}.
            \]
    \end{enumerate}
\end{definition}

\begin{definition}[{\cite[p.~5]{siegel:on}}]
    Given a set of games $\mathcal{A}$, define the \emph{universal closure} of
    $\mathcal{A}$, denoted $\mathcal{D}(A)$, as the minimal universe containing
    $\mathcal{A}$ (with respect to set inclusion).
\end{definition}

The reader may verify that intersections of universes are themselves universes,
which Siegel also notes, and hence writing `\emph{the} minimal universe
containing $\mathcal{A}$' is justified. Take caution with the notation
$\mathcal{D}(\mathcal{U})$; merely taking the dicotic closure of a set of games
is not enough to form a universe. It will, however, precisely when the set of
ends of the set already satisfies every other closure property in
\cref{def:universe}, and if every other game in the set lies within the dicotic
closure of its set of ends.

As mentioned in the introduction, we prove a striking property of universes:
they are completely characterised by their (Left) ends. That is, we can
reconstruct the forms of a universe given just its set of Left ends. We now
state and prove this result, in a slightly more general form than we have just
described.

\begin{proposition}
    \label{prop:left-ends}
    If $\mathcal{A}$ is a set of games that is conjugate closed and hereditary,
    with $E_L(\mathcal{A})$ its set of Left ends, then
    $\mathcal{D}(\mathcal{A})=\mathcal{D}(E_L(\mathcal{A}))$.
\end{proposition}

\begin{proof}
    Since $E_L(\mathcal{A})\subseteq\mathcal{A}$, it is clear that
    $\mathcal{D}(E_L(\mathcal{A}))\subseteq\mathcal{D}(\mathcal{A})$. Suppose,
    for a contradiction, that we do not have equality. Then there must exist
    some game $G$ of minimal formal birthday such that
    $G\in\mathcal{D}(\mathcal{A})\backslash\mathcal{D}(E_L(\mathcal{A}))$.
    Since $\mathcal{A}$ is conjugate closed, it follows that
    $\mathcal{D}(\mathcal{A})$ and $\mathcal{D}(E_L(\mathcal{A}))$ have
    identical ends; so, $G$ cannot be an end. By the heredity of $\mathcal{A}$,
    every option $G'$ of $G$ is also in $\mathcal{A}$; and since $G'$ has
    formal birthday strictly less than $G$, we know that
    $G'\in\mathcal{D}(E_L(\mathcal{A}))$ by minimality of $G$. Thus, by the
    dicotic closure property of $\mathcal{D}(E_L(\mathcal{A}))$, we obtain that
    $G\in\mathcal{D}(E_L(\mathcal{A}))$, which is the contradiction we were
    after.
\end{proof}

Note that, if $\mathcal{U}$ is a universe, then it is already conjugate closed
and hereditary by definition, and $\mathcal{D}(\mathcal{U})$ is of course equal
to $\mathcal{U}$ itself. As such, the above proposition tells us exactly how to
recover the forms of $\mathcal{U}$ from its set of Left ends: take the
universal closure.

As noted in the introduction, understanding the structure of arbitrary Left
ends appears no different from understanding the structure of $\mathcal{M}$.
But we are gifted with the existence of a very interesting subset of Left ends:
the \emph{Left dead ends}. How useful it would be, then, if we had some
machinery to work with these dead ends. Fortunately, Siegel has given us just
that in his observation of a simplification in the comparison test of Larsson
et al.\ \cite[Theorem 4 on p.~6]{larsson.nowakowski.ea:absolute}.

\begin{theorem}[{\cite[Theorem 3.1 on p.~7]{siegel:on}}]
    \label{thm:reduced}
    If $\mathcal{U}$ is a universe, and $G$ and $H$ are Left dead ends, then
    $G\geq_\mathcal{U}H$ if and only if the following hold:
    \begin{enumerate}
        \item
            if $G\cong0$, then $H\cong0$; otherwise,
        \item
            for every $G^R\in G^{\mathcal{R}}$, there exists some $H^R\in
            H^{\mathcal{R}}$ such that $G^R\geq_\mathcal{U}H^R$.
    \end{enumerate}
\end{theorem}

As Siegel points out, this means that comparing Left dead ends is independent
of the universe; the partial order of Left dead ends is \emph{multiversal}.

\begin{corollary}[\cite{siegel:on}]
    \label{cor:multiversal}
    If $\mathcal{U}$ is a universe, and $G$ and $H$ are Left dead ends with
    $G\geq_\mathcal{U}H$, then $G\geq_\mathcal{W}H$ for all universes
    $\mathcal{W}$.
\end{corollary}

We will remind the reader here that, when we write `$\geq$' (and similarly for
`$=$'), we mean modulo $\mathcal{M}$ (the full mis\`ere universe). Given
\cref{cor:multiversal}, however, if we are discussing Left dead ends, then
$\geq$ and $\geq_\mathcal{U}$ are the same relation for every universe
$\mathcal{U}$.

As our study will exist largely within an algebraic context, we take the time
now to recall some relevant definitions. The reader is invited to consult the
factorisation literature for more detail, but one should keep in mind that it
is common there to define a monoid as a commutative and unit-cancellative
(sometimes even cancellative) semigroup with identity, which is non-standard in
other areas of algebra. In particular, see
\cite[pp.~7--9]{geroldinger:additive} for a sufficient and concise
introduction, or any of: \cite[Ch.\ 1]{geroldinger.halter-koch:non-unique},
\cite[pp.~207--211]{geroldinger.halter-koch:survey}, and \cite[\S2.2 on
p.~27]{geroldinger.zhong:factorization}.

While multiplicative notation can commonly be used, we will write the
background here in additive notation; in part because we will be working
exclusively with a commutative monoid, but mainly because additive notation is
so prevalent in Combinatorial Game Theory, which is the main field in which our
study resides.

The algebraic object of importance to us here is the commutative pomonoid (also
called an \emph{ordered monoid}, as in \cite[p.~193]{blyth:lattices}), which
(for us) is a semigroup $S$ equipped with: a (unique) identity element `0'; a
commutative binary operation `+'; and a partial order `$\geq$' compatible with
the binary operation:
\[
    a\geq b\implies a+c\geq b+c\quad\text{for all }a,b,c\in S\cup\{0\}.
\]
Of course, a pomonoid is trivially also a monoid. Writing $M$ for a monoid,
recall that $M$ is said to be:
\begin{itemize}
    \item
        \emph{reduced} if its set of invertible elements contains only a single
        element (the identity);
    \item
        \emph{cancellative} if, for all $a,b,c\in M$, the implication
        $a+c=b+c\implies a=b$ holds; and
    \item
        \emph{pocancellative} if $M$ is a pomonoid and, for all $a,b,c\in M$,
        the implication $a+c\geq b+c\implies a\geq b$ holds.
\end{itemize}

Note, of course, that, if every element of $M$ is invertible, then $M$ is a
group. Also, it is immediate that pocancellativity implies cancellativity.

An important property of monoids to us will be atomicity. We say that a
non-invertible element $a$ of a monoid $M$ is an \emph{atom} if, whenever
$a=b+c$ for some elements $b,c\in M$, then either $b$ or $c$ is invertible. In
the case that $M$ is reduced, recall that an element being invertible means
that it is the identity.

An element $m\in M$ is called \emph{atomic} if it can be written as a (finite)
sum of atoms; i.e.\ $m=a_1+a_2+\cdots+a_n$ for some atoms $a_i\in M$. We call
such a sum of atoms a \emph{factorisation} of $m$ (for a more formal
definition, consult one of the references provided). The \emph{length} of a
factorisation is the number of atoms used (counting duplicates):
$a_1+a_2+\cdots+a_n$ has length $n$, regardless of whether the $a_i$ are
pairwise distinct or not. We will say that an element $m\in M$ is
\emph{uniquely factorisable} if it admits exactly one factorisation (up to a
rearrangement). The monoid $M$ itself is called \emph{atomic} if every
(non-invertible) element is atomic. It will be a convenience for us to
introduce a term for a non-invertible element that is atomic but not an atom:
\emph{molecule}\footnote{This term was suggested by Danny Dyer.}.

\begin{definition}
    In a monoid $M$, we call a non-invertible element $m\in M$ a
    \emph{molecule} if $m$ is atomic but not an atom.
\end{definition}

A monoid is called a \emph{BF-monoid} (a \emph{bounded factorisation monoid})
if it is atomic and every element admits a longest factorisation. It is called
an \emph{FF-monoid} (a \emph{finite factorisation monoid}) if it is a BF-monoid
and every element admits only finitely many factorisations.

It would be impossible for us to talk about factorisations here without
discussing primality. Given elements $a$ and $b$ in a monoid $M$, we say that
$a$ \emph{divides} $b$, written $a\mid b$, if there exists some $c\in M$ such
that $a+c=b$. A non-invertible element $p\in M$ is then called \emph{prime} if,
whenever $p\mid a+b$ for some $a,b\in M$, then $p\mid a$ or $p\mid b$. We will
call a factorisation consisting only of prime elements a \emph{prime
factorisation}; and we will say that an element is \emph{prime-factorisable} if
it admits such a factorisation.

A particular kind of atom that has received attention in the literature is that
of the \emph{strong} atom \cite[p.~1]{angermuller:strong} (somtimes called
\emph{absolutely irreducible}, as in \cite[Definition 7.1.3 on
p.~444]{geroldinger.halter-koch:non-unique}); these are atoms $a$ such that the
only divisors of $n\cdot a$ are multiples of $a$ (up to a unit).

It is a simple result that, in an atomic monoid, every prime is an atom. It is
not the case, in general, that every atom is prime. A key question we will ask
later, about the monoid of Left dead ends, is when we can conclude that an atom
is prime. Furthermore, in a monoid that is reduced, atomic, cancellative, and
commutative, every prime-factorisable element is uniquely factorisable. Thus,
in such a monoid, there is a clear dichotomy: every element has a
factorisation, but not necessarily uniquely; conversely, not every element has
a prime factorisation, but, when it does, it is uniquely factorisable.

There are many notations and bits of terminology used to describe the sets of
lengths of factorisations of elements of a monoid, but we will not have use for
all of that here. We borrow only a notation to denote the length of a
\emph{longest} factorisation of an element $m$: $\longest{m}$.

Note that any universe of combinatorial games is trivially a commutative
pomonoid (and there are many partial orderings that one might want to equip).
So, too, is any set of games containing $0$ that is additively closed.

It will be convenient for us to use sumset notation (Minkowski sums), writing
$A+B$ for the set $\{a+b:(a,b)\in A\times B\}$, and also $A+k$ for $\{a+k:a\in
A\}$ in the special case of adding a singleton.

\section{Properties of the monoid}
\label{sec:properties}

We use $\mathcal{L}$ to denote the set of all Left dead ends, in accordance
with Siegel \cite[Definition 1.2 on p.~2]{siegel:on}. It is clear that
$(\mathcal{L},+)$ is a commutative monoid (since it is a set of games
containing $0$ that is additively closed). We will abuse notation and also
refer to this monoid by $\mathcal{L}$. We remind the reader again that `$\geq$'
and `=' will have their usual meaning (modulo full mis\`ere, which is
equivalent to being modulo any universe thanks to \cref{cor:multiversal}).

We now prove a number of properties of the monoid $\mathcal{L}$ that we
discussed in the introduction and preliminaries: that $\mathcal{L}$ is reduced,
pocancellative, and an FF-monoid. Along the way, we will build some other tools
that will also help us later, and so divide this path up into subsections.

\subsection{Basic properties}

We begin by showing that the monoid of Left dead ends is reduced, which follows
immediately after observing that $\mathcal{L}$ is a submonoid of the reduced
monoid $\mathcal{M}$---recall that Mesdal and Ottaway showed in \cite[Theorem 7
on p.~5]{mesdal.ottaway:simplification} that, for any partizan game form $G$ in
mis\`ere play, $G=0$ if and only if $G\cong0$.

\begin{theorem}
    The monoid of Left dead ends $\mathcal{L}$ is reduced.
\end{theorem}

\begin{proof}
    A partizan game form is equal to 0 if and only if it is isomorphic to 0,
    and hence $0$ is the only invertible element present in $\mathcal{L}$.
    Thus, $\mathcal{L}$ is reduced.
\end{proof}

Recall from Siegel \cite[Definition 1.4 on p.~9]{siegel:combinatorial} that a
run of length $k$ in a combinatorial game $G$ is a sequence
$(G_0,G_1,\dots,G_k)$ such that $G_0\cong G$ and $G_{i+1}$ is an option of
$G_i$ for every $i$.\footnote{
    The definition cited actually writes $G_0=G$ instead of $G_0\cong G$, but
    the latter is what is meant, so we write it to help avoid confusion.
} It is clear that, if $G$ is a short game, then there exists a run of length
$k$ exactly for those $k\in\mathbb{N}_0$ satisfying $0\leq
k\leq\tilde\birth(G)$. Then, as a simple corollary of \cref{thm:reduced}, we
have the following result.

\begin{corollary}
    \label{cor:move-runs}
    If $G$ and $H$ are Left dead ends with $G\geq H$, and $k$ is an integer
    satisfying $0\leq k\leq\tilde\birth(G)$, then, for any run
    $(G_0,G_1,\dots,G_k)$ of $G$ of length $k$, there exists a run
    $(H_0,H_1,\dots,H_k)$ of $H$ of length $k$ such that $G_i\geq H_i$ for
    every $0\leq i\leq k$.
\end{corollary}

\begin{proof}
    Take a run $(G_0,G_1,\dots,G_k)$ of $G$ of length $k$. By hypothesis, we
    have $G_0\geq H$. Then, by \cref{thm:reduced}, we know that there exists an
    option $H_1$ of $H$ such that $G_1\geq H_1$. By induction on the formal
    birthday, given the run $(G_1,\dots,G_k)$ of $G_1$ of length $k-1$, there
    exists a run $(H_1,\dots,H_k)$ of $H_1$ of length $k-1$ such that $G_i\geq
    H_i$ for every $1\leq i\leq k$. Thus, the run $(H,H_1,\dots,H_k)$ of $H$ of
    length $k$ satisfies the requirements of the conclusion, yielding the
    result.
\end{proof}

It will simplify our later exposition if we give the following definition of a
\emph{good} option. Without it, we would have to make copious references to the
canonical form of a game, which the reader is free to do themselves if they so
wish.

\begin{definition}
    \label{def:good-option}
    If $G$ is a combinatorial game, then we will say that a (Left, respectively
    Right) option $G'$ of $G$ is \emph{good} if it is not strictly dominated by
    another (Left, respectively Right) option of $G$.
\end{definition}

So, an option $G'$ of a Left dead end $G$ is good if there exists no option
$G''$ with $G'>G''$. Equivalently, $G'$ is good if it is equal to some option
that appears in the canonical form of $G$; this follows from \cite[Theorem 3.7
on p.~8]{siegel:on}, which establishes that the canonical form of a Left dead
end is itself a Left dead end, and is obtained by removing all dominated
options from all subpositions of the game.

Given \cref{def:good-option}, we will state another obvious corollary of
\cref{thm:reduced} that we will have use for later on. Note, of course, that
every non-zero Left dead end has at least one good option.

\begin{corollary}
    \label{cor:good-option-replies}
    If $G$ and $H$ are Left dead ends with $G\geq H$, and $G'$ is a good option
    of $G$, then there exists a (good) option $H'$ of $H$ such that $G'\geq
    H'$. In particular, if $G=H$, then $G'=H'$.
\end{corollary}

\begin{proof}
    By \cref{thm:reduced}, there exists some option $H''$ of $H$ such that
    $G'\geq H''$. By \cref{def:good-option}, there exists a good option $H'$ of
    $H$ (possibly $H''$ itself) such that $H''\geq H'$. By transitivity of the
    relation, it is immediate that $G'\geq H'$.

    If $G=H$, then, in addition to our previous argument, we know also that
    $H\geq G$, and so there must exist some good option $G''$ of $G$ such that
    $G'\geq H'\geq G''$. Since $G'$ is a good option by hypothesis, it follows
    that $G'=G''$ straight from \cref{def:good-option}. Thus, $G'=H'$.
\end{proof}

It is unknown for most mis\`ere monoids whether they exhibit cancellativity or
not. Luckily for us, Left dead ends are very well-behaved, and we prove shortly
that the monoid of Left dead-ends is pocancellative. This would also follow
from recent work by the present author, McKay, Milley, Nowakowski, and Santos
\cite{davies.mckay.ea:pocancellation}, which shows that $\mathcal{E}$ is
pocancellative---whereas $\mathcal{D}$ and $\mathcal{M}$ are not. We do not
need all this machinery, however, and we include a concise proof here for our
special case.

\begin{lemma}
    \label{lem:pocancellative}
    If $\mathcal{U}$ is a universe, $G,H\in\mathcal{U}$ are Left dead ends, and
    $J$ is a pocancellative element of $\mathcal{U}$ satisfying $G+J\geq H+J$,
    then $G\geq H$.
\end{lemma}

\begin{proof}
    Observe that $G+J\geq H+J$ implies $G+J\geq_\mathcal{U}H+J$ by
    \cref{cor:multiversal}. Since $J$ is pocancellative in $\mathcal{U}$, it
    follows that $G\geq_\mathcal{U}H$, yielding the result via
    \cref{cor:multiversal}.
\end{proof}

It is not necessary for us here, but an analogous lemma holds where
pocancellativity is replaced with cancellativity, whose proof is almost
identical. We simply state it below.

\begin{lemma}
    If $\mathcal{U}$ is a universe, $G,H\in\mathcal{U}$ are Left dead ends, and
    $J$ is a cancellative element of $\mathcal{U}$ satisfying $G+J=H+J$, then
    $G=H$.
\end{lemma}

Recalling a result of Milley and Renault \cite[Theorem 6 on
p.~2226]{milley.renault:dead} that says Left dead ends are invertible in
$\mathcal{E}$ (the dead ending universe), we now have the tools to show that
our monoid is pocancellative.

\begin{theorem}
    \label{thm:pocancellative}
    The monoid of Left dead ends $\mathcal{L}$ is pocancellative.
\end{theorem}

\begin{proof}
    The result follows immediately from \cref{lem:pocancellative} after
    remarking that Left dead ends are invertible, and hence pocancellative, in
    $\mathcal{E}$.
\end{proof}

While we have used the invertibility of Left dead ends in $\mathcal{E}$ to
apply \cref{lem:pocancellative}, it is unclear how, in an arbitrary universe,
pocancellativity differs from invertibility (the latter, of course, always
implies the former). See \cref{app:pocancellativity} for a brief discussion.

\subsection{The group of differences}

We take a very brief detour to prove a statement about the group of differences
of our monoid. Recall that, given a commutative monoid, we can always construct
an abelian group by simply introducing a formal inverse for each element; this
group is called the \emph{group of differences}, or the \emph{Grothendiek
group}, of the monoid. Since our monoid of Left dead ends is also cancellative
(\cref{thm:pocancellative}), we in fact know that it must embed in its group of
differences.

\begin{proposition}
    \label{prop:grothendiek}
    If $G$ and $H$ are Left dead ends with $n\cdot G=n\cdot H$ for some $n>0$,
    then $G=H$.
\end{proposition}

\begin{proof}
    If $n=1$, then the result is immediate, so assume now that $n>1$.

    If $G=0$, then $n\cdot0=0=H$ yields the result immediately. Otherwise, by
    \cref{cor:good-option-replies}, there exist options $G'$ and $H'$ such that
    \begin{equation}
        \tag{1}
        \label{eq:temp}
        G'+(n-1)\cdot G=H'+(n-1)\cdot H.
    \end{equation}
    Multiplying by $n$ yields
    \[
        n\cdot G'+n(n-1)\cdot G=n\cdot H'+n(n-1)\cdot H.
    \]
    Since $n\cdot G=n\cdot H$ by hypothesis, it is clear that $n(n-1)\cdot
    G=n(n-1)\cdot H$, and hence it follows by cancellativity
    (\cref{thm:pocancellative}) that $n\cdot G'=n\cdot H'$. By induction, we
    have $G'=H'$. It now follows from \cref{eq:temp} that $(n-1)\cdot
    G=(n-1)\cdot H$, from which it follows by induction again that $G=H$, which
    is the result.
\end{proof}

This result, and indeed its proof, is very similar to a result of Siegel's in
the impartial mis\`ere monoid; in particular, see \cite[Theorem 29 on
p.~16]{siegel:impartial} for Siegel's result, and
\cref{app:impartial-similarity} for a further discussion on the similarities
between our two monoids.

\begin{corollary}
    The Grothendiek group (i.e.\ the group of differences) of the monoid of
    Left dead ends is torsion free.
\end{corollary}

\begin{proof}
    Write $\mathcal{D}$ for the Grothendiek group. Let $G-H\in\mathcal{D}$, and
    suppose that $n\cdot(G-H)=0$ for some $n$. Then $n\cdot G=n\cdot H$, and we
    have by \cref{prop:grothendiek} that $G=H$, and hence $G-H=0$.
\end{proof}

As such, the Grothendiek group of our monoid of Left dead-ends is necessarily a
(countable) torsion-free abelian group of countably infinite rank.

For the impartial mis\`ere monoid, Siegel proved that $*$ is the only torsion
element of the group of differences \cite[Corollary 30 on
p.~16]{siegel:impartial}. As in the impartial case, we leave open the following
problem.

\begin{problem}
    What is the isomorphism type of the group of differences of the monoid of
    Left dead ends?
\end{problem}

\subsection{Terminal lengths}

As it stands, we have shown that our commutative pomonoid $\mathcal{L}$ is
reduced and pocancellative. Before we show that it is also an FF-monoid, we
first develop some useful notions.

\begin{definition}
    If $G$ is a combinatorial game and $(G_0,G_1,\dots,G_k)$ is a run of $G$
    such that $G_k\cong0$, then we say the run is \emph{terminal}.
\end{definition}

This is not to be confused with a \emph{play} (cf.\ \cite[Definition 1.4 on
p.~9]{siegel:combinatorial}); we are not requiring the run to be alternating,
and we are requiring it to terminate at 0 rather than there just being no
options for the player whose turn it is.

\begin{definition}
    \label{def:terminal-lengths}
    If $G$ is a short game, then we define the set of \emph{terminal lengths}
    of $G$, denoted $\terminal(G)$, as the set of all $n$ such that there
    exists a terminal run of $G$ of length $n$.
\end{definition}

Note that the set of terminal lengths is defined on the strict form of $G$, and
not on its equivalence class modulo equality. For Left dead ends, however, we
will soon see in \cref{thm:subset-terminal-lengths} that it is indeed
well-defined up to equality. It is useful to remark that no game can have empty
terminal lengths; the empty game satisfies $\terminal(0)=\{0\}$.

Of course, we could have also defined terminal lengths recursively in the
following way, which will also provide a useful viewpoint. This is entirely
analogous to how one could define the formal birthday $\tilde{\birth}(G)$ of a
short game $G$ either as the height of the game tree, or recursively as one
more than the maximum formal birthday among its options (with the base case
$\tilde{\birth}(0)\coloneq0$). Recall that we are using Minkowski sums (sumset
notation) here.

\begin{lemma}
    \label{lem:terminal-subset}
    If $G$ is a non-zero Left dead end, then
    \[
        \terminal(G)=\bigcup(\terminal(G')+1),
    \]
    where $G'$ ranges over all options of $G$. In particular, it holds that
    $\terminal(G')\subseteq\terminal(G)-1$ for every option $G'$ of $G$.
\end{lemma}

\begin{proof}
    Observing that $(G_0,G_1,\dots,G_n)$ is a run of $G$ if and only if $G_1$
    is an option of $G_0$ and $(G_1,\dots,G_n)$ is a run of $G_1$ yields the
    result immediately.
\end{proof}

For any short game $G$, we clearly have the identity
$\tilde\birth(G)=\max(\terminal(G))$. We introduce an opposing notion that will
oft be useful in our study of Left dead ends: the \emph{race} of a game, which
measures how quickly one can \emph{race} to the finish.\footnote{It is
    interesting to compare this idea with \cite[Definition 2 on
    p.~2225]{milley.renault:dead}.
}

\begin{definition}
    \label{def:race}
    If $G$ is a short game, define the \emph{race} of $G$, denoted $\race(G)$,
    as
    \[
        \race(G)\coloneq\min(\terminal(G)).
    \]
\end{definition}

Of course, we could have defined $\race(G)$ recursively as one more than the
minimum race amongst its options (with the base case $\race(0)\coloneq0$), and
the reader is invited to confirm that this is equivalent to the definition
given above. For a brief discussion of why we have not created a split for race
and \emph{formal} race, like there is for birthday and formal birthday, see
\cref{app:formal-race}.

We will say that a Left dead end is an \emph{integer} if it is isomorphic
either to $0\coloneq\{\cdot\mid\cdot\}$ or to
$\overline{n}\coloneq\{\cdot\mid\overline{n-1}\}$ for some $n\geq1$ (where
$\overline{0}\coloneq0$). It is then clear that
$\race(\overline{n})=n=\tilde{\birth}(\overline{n})$. The following lemma
captures this observation and will prove useful later.

\begin{lemma}
    \label{lem:integer}
    A Left dead end $G$ is an integer if and only if $|\terminal(G)|=1$. In
    particular, $G=\overline{n}$ if and only if $\terminal(G)=\{n\}$.
\end{lemma}

\begin{proof}
    The only terminal run of $\overline{n}$ is
    $(\overline{n},\overline{n-1},\dots,0)$, which has length $n$. Thus, if $G$
    is an integer, then $|\terminal(G)|=1$.

    Now suppose instead that $|\terminal(G)|=1$; say $\terminal(G)=\{n\}$. If
    $n=0$, then $G=0$. Otherwise, by \cref{lem:terminal-subset}, every option
    $G'$ of $G$ satisfies $\terminal(G')=\{n-1\}$. Hence $G'=\overline{n-1}$ by
    induction, yielding that $G=\overline{n}$.
\end{proof}

Another family of Left dead ends that will be important to us is the class of
\emph{waiting games} (see \cite[Definition 3 on
p.~9]{larsson.milley.ea:recursive}, but note the different name and notation):
we write $W_n\coloneq\{\cdot\mid0,W_{n-1}\}$ for the waiting game of rank $n$
for every $n>0$, with $W_0\coloneq0$. It is a simple observation that
$\terminal(W_n)=\{t\in\mathbb{N}:1\leq t\leq n\}$ for every $n>0$.

To confirm the reader's understanding up to this point, they may wish to verify
that the set of terminal lengths of $G=\{\cdot\mid W_2,\overline{3}\}$ is the
set $\terminal(G)=\{2,3,4\}$. In particular, $\race(G)=2$ and
$\tilde{\birth}(G)=4$.

We state some obvious properties of terminal lengths relating to 0 that we will
refer back to later.

\begin{lemma}
    \label{lem:0-props}
    If $G$ is a Left dead end, then the following are equivalent:
    \begin{multicols}{3}
        \begin{enumerate}
            \item
                $G=0$;
            \item
                $G\cong0$;
            \item
                $0\in\terminal(G)$;
            \item
                $\race(G)=0$;
            \item
                $\birth(G)=0$.
        \end{enumerate}
    \end{multicols}
\end{lemma}

\begin{proof}
    We have already seen that (1) and (2) are equivalent by a result of Mesdal
    and Ottaway \cite[Theorem 7 on p.~5]{mesdal.ottaway:simplification}.
    Additionally, it is trivial that (1) and (5) are equivalent. Clearly
    $(2)\implies(3)$ and $(4)$. Then, $(3)\implies(4)$ by definition of race.

    To complete the proof, it suffices to show that $(4)\implies(2)$. But if
    $\race(G)=0$, then no moves can be made, and so immediately we have
    $G\cong0$.
\end{proof}

We now prove a powerful and surprising implication of two Left dead ends being
comparable, namely that their terminal lengths are also comparable (by
inclusion).

\begin{theorem}
    \label{thm:subset-terminal-lengths}
    If $G$ and $H$ are Left dead ends with $G\geq H$, then
    $\terminal(G)\subseteq\terminal(H)$.
\end{theorem}

\begin{proof}
    If $G\cong0$, then $H\cong0$ by \cref{thm:reduced}, and we have the result.
    Otherwise, again by \cref{thm:reduced}, for every option $G^R$, there
    exists some option $H^R$ with $G^R\geq H^R$. By induction, we must have
    $\terminal(G^R)\subseteq\terminal(H^R)$. Then, by
    \cref{lem:terminal-subset}, observe
    \begin{align*}
        \terminal(G)&=\bigcup(\terminal(G')+1)\\
                    &\subseteq\bigcup(\terminal(H')+1)\\
                    &=\terminal(H),
    \end{align*}
    where $G'$ and $H'$ range over all options of $G$ and $H$ respectively,
    which yields the result.
\end{proof}

For an alternative proof that constructs an explicit distinguishing game, see
\cref{app:alt-proof}.

\begin{corollary}
    \label{cor:comparable-race-birth}
    If $G$ and $H$ are Left dead ends with $G\geq H$, then
    $\race(G)\geq\race(H)$, and $\tilde\birth(G)\leq\tilde\birth(H)$. Hence
    also $\birth(G)\leq\birth(H)$.
\end{corollary}

\begin{proof}
    Since $\terminal(G)\subseteq\terminal(H)$ by
    \cref{thm:subset-terminal-lengths}, it is clear from \cref{def:race} that
    \begin{align*}
        \race(G)&=\min(\terminal(G))\\
                &\leq\min(\terminal(H))\\
                &=\race(H).
    \end{align*}
    Similarly,
    \begin{align*}
        \tilde\birth(G)&=\max(\terminal(G))\\
                       &\leq\max(\terminal(H))\\
                       &=\tilde\birth(H),
    \end{align*}
    from which it follows that all Left dead ends in a given equivalence class
    must have equal formal birthday, and hence so does the canonical form
    (which must also be a Left dead end). Thus, $\birth(G)\leq\birth(H)$.
\end{proof}

Of course, if we have two Left dead ends $G$ and $H$ with $G=H$, then
\cref{thm:subset-terminal-lengths} tells us that $\terminal(G)=\terminal(H)$,
and \cref{cor:comparable-race-birth} says in particular that
$\race(G)=\race(H)$ and $\tilde\birth(G)=\tilde\birth(H)$. It is from this
point onwards that we will no longer refer to the formal birthday of a game,
and will instead use only its birthday, since \cref{cor:comparable-race-birth}
shows that these two functions agree on Left dead ends. The need for a
distinction between birthday and formal birthday in other contexts in CGT is
precisely because of the lack of this well-definedness up to equality there.

Siegel also proved the special case of \cref{cor:comparable-race-birth} as it
pertains to the birthday specifically; see \cite[Propositions 3.3 and 3.4 on
p.~7]{siegel:on}.

A weakened converse of \cref{thm:subset-terminal-lengths} is explored in
\cref{prop:flex-1-minimal}. Unsurprisingly, two Left dead ends having equal
terminal lengths does not imply the two games are equal. Take, for example,
$G=\overline{1}+W_2$ and $H=\{\cdot\mid\overline{1},\overline{2}\}$, where both
games have terminal lengths $\{2,3\}$, but $H>G$. \cref{prop:flex-1-minimal},
however, will allow us to conclude equality for exactly those Left dead ends
that are weakly 1-flexible (see \cref{def:flex}).

We now prove another simple consequence of \cref{thm:subset-terminal-lengths}
that combines itself with \cref{cor:move-runs} to provide a simple tool that
can often be useful when trying to show that $G\not\geq H$ for some Left dead
ends $G$ and $H$.

\begin{corollary}
    \label{cor:terminal-comparison}
    If $G$ and $H$ are Left dead ends with $G\geq H$, and $k$ is an integer
    satisfying $0\leq k\leq\tilde\birth(G)$, then, for any run
    $(G_0,G_1,\dots,G_k)$ of $G$ of length $k$, there exists a run
    $(H_0,H_1,\dots,H_k)$ of $H$ of length $k$ such that
    $\terminal(G_i)\subseteq\terminal(H_i)$ for every $0\leq i\leq k$.
\end{corollary}

\begin{proof}
    By \cref{cor:move-runs}, we know that there exists a run
    $(H_0,H_1,\dots,H_k)$ of $H$ such that $G_i\geq H_i$ for every $0\leq i\leq
    k$. \cref{thm:subset-terminal-lengths} then yields the result immediately.
\end{proof}

It is tempting to try and prove the converse of this statement. But, alas, it
is not true in general. For example, take
\begin{gather*}
    G=\{\cdot\mid\{\cdot\mid W_2,\{\cdot\mid\overline{1},
    \overline{2}\}\}\}\quad\text{and}\\
    H=\{\cdot\mid\{\cdot\mid W_2, \overline{3}\}, \{\cdot\mid\overline{1},
    \{\cdot\mid\overline{1}, \overline{2}\}\}\},
\end{gather*}
where the reader may verify that $G\not\geq H$, but, for every run
$(G_0,G_1,\dots,G_k)$ of $G$, there exists a run $(H_0,H_1,\dots,H_k)$ such
that $\terminal(G_i)\subseteq\terminal(H_i)$ for every $0\leq i\leq k$.

It will be prudent of us here to remark upon the terminal lengths of sums of
games.

\begin{proposition}
    \label{prop:terminal-length-sums}
    If $G$ and $H$ are Left dead ends, then
    \[
        \terminal(G+H)=\terminal(G)+\terminal(H).
    \]
    In particular, $\race(G+H)=\race(G)+\race(H)$ and
    $\birth(G+H)=\birth(G)+\birth(H)$.
\end{proposition}

\begin{proof}
    When playing the sum $G+H$, observe that $G+H$ is terminated if and only if
    $G$ and $H$ are both terminated. Thus, the only way to terminate $G+H$ is
    by terminating $G$ in any $a\in\terminal(G)$ number of moves and
    terminating $H$ in any $b\in\terminal(H)$ number of moves.

    Finally, observe that
    \begin{align*}
        \race(G+H)&=\min(\terminal(G+H))\\
                  &=\min(\terminal(G))+\min(\terminal(H))\\
                  &=\race(G)+\race(H),
    \end{align*}
    and similarly for the birthday.
\end{proof}

Siegel proved the special case of $\birth(G+H)=\birth(G)+\birth(H)$ in
\cite[Proposition 3.8 on p.~8]{siegel:on}.

\begin{corollary}
    \label{cor:terminal-length-sum-size}
    If $G$ and $H$ are Left dead ends, then
    \[
        |\terminal(G+H)|\geq\max(|\terminal(G)|,|\terminal(H)|),
    \]
    with equality if and only if at least one of $G$ and $H$ is an integer.
\end{corollary}

\begin{proof}
    Suppose, without loss of generality, that $G$ is an integer. It follows
    from \cref{lem:integer} that $\terminal(G)=\{n\}$ for some $n$.
    \cref{prop:terminal-length-sums} then yields
    \begin{align*}
        |\terminal(G+H)|&=|\{n+b:b\in\terminal(H)\}|\\
                &=|\terminal(H)|\\
                &=\max(|\terminal(G)|,|\terminal(H)|).
    \end{align*}

    Suppose now that neither $G$ nor $H$ is an integer. So, without loss of
    generality, $1<|\terminal(G)|\leq|\terminal(H)|$ by \cref{lem:integer}.
    Then, the union
    \[
        (\race(G)+\terminal(H))\cup\{\birth(G+H)\}\subseteq\terminal(G+H)
    \]
    has cardinality $|\terminal(H)|+1$. Recalling that
    $|\terminal(H)|=\max(|\terminal(G)|,|\terminal(H)|)$, we obtain the result.
\end{proof}

\cref{cor:terminal-length-sum-size} can actually also be a viewed as a result
of additive number theory; in particular, see \cite[Theorem 1.4 on
p.~8]{nathanson:additive}. It might be worthwhile to investigate these
relationships further.

\subsection{Finite factorisation monoid}
\label{subsec:FFM}

We now return more resolutely to proving that the monoid of Left dead ends is
an FF-monoid, having amassed more than enough firepower.

\begin{lemma}
    \label{lem:atom-birthday}
    If $G,H$, and $K$ are non-zero Left dead ends with $G=H+K$, then
    $\birth(H)<\birth(G)$.
\end{lemma}

\begin{proof}
    Suppose, for a contradiction, that $\birth(H)\geq\birth(G)$. It then
    follows from \cref{prop:terminal-length-sums} that
    $\birth(H+K)=\birth(H)+\birth(K)>\birth(G)$, contradicting
    \cref{cor:comparable-race-birth}.
\end{proof}

\begin{lemma}
    \label{lem:atomic}
    The monoid of Left dead ends $\mathcal{L}$ is atomic.
\end{lemma}

\begin{proof}
    We will proceed by induction on the birthday of a Left dead end. The empty
    game trivially has a factorisation: the empty sum. Suppose all Left dead
    ends of birthday strictly less than $n>0$ admit factorisations. Now choose
    an arbitrary Left dead end $G$ with $\birth(G)=n$. If $G$ cannot be written
    in the form $H+K$, for some non-zero Left dead ends $H$ and $K$, then $G$
    is itself an atom by definition. Otherwise, it follows from
    \cref{lem:atom-birthday} that $\birth(H),\birth(K)<n$. By induction, $H$
    and $K$ admit factorisations, and, thus, so does $G$.
\end{proof}

\begin{lemma}
    \label{lem:bf}
    The monoid of Left dead ends $\mathcal{L}$ is a BF-monoid. In particular,
    if $G$ is a Left dead end, then the length of each factorisation of $G$ is
    at most $\race(G)$; that is, $\longest{G}\leq\race(G)$.
\end{lemma}

\begin{proof}
    Recall that $\mathcal{L}$ is atomic (\cref{lem:atomic}), and let
    $G_1+\dots+G_n$ be an arbitrary factorisation of $G$. By
    \cref{prop:terminal-length-sums}, it follows that
    $\race(G)=\race(G_1)+\dots+\race(G_n)\geq n$. Thus, there can be at most
    $\race(G)$ atoms in a factorisation of $G$, or else
    \cref{cor:comparable-race-birth} would be violated.
\end{proof}

While the bound in \cref{lem:bf} is sharp, we will have more to say in
\cref{sec:factorisation-lengths} on improving this for many families; see
\cref{thm:one-option-bound,thm:length-bound}.

\begin{corollary}
    \label{cor:race-1-atom}
    If $G$ is a Left dead end with $\race(G)=1$, then $G$ is an atom.
\end{corollary}

\begin{proof}
    By \cref{lem:bf}, the length of each factorisation of $G$ is at most 1,
    yielding the result immediately.
\end{proof}

Another way of phrasing this is that, if $G$ has an option to the empty game,
then $G$ is an atom. Siegel calls these games \emph{terminable}, and
essentially also proves this proposition, albeit in a different language and
context, in \cite[Proposition 3.9 on p.~8]{siegel:on}. For example, $W_n$ is a
(terminable) atom for every positive $n$. This result (\cref{cor:race-1-atom})
will be strengthened further in \cref{prop:integer-option-atom}.

Finally, we prove the main result of the section.

\begin{theorem}
    \label{thm:ff}
    The monoid of Left dead ends $\mathcal{L}$ is an FF-monoid.
\end{theorem}

\begin{proof}
    Let $G$ be a Left dead end. Recall that $\mathcal{L}$ is a BF-monoid
    (\cref{lem:bf}). By \cref{lem:atom-birthday}, it follows that each atom in
    each factorisation of $G$ must have birthday at most $\birth(G)$ (there is
    equality precisely when $G$ is itself an atom). There are finitely many
    short game forms of any given birthday, and hence only finitely many atoms
    of birthday at most $\birth(G)$. \cref{lem:bf} tells us that the set of
    lengths of all factorisations of $G$ is finite. Each factorisation of $G$
    is thus obtained by choosing a finite number of atoms from a finite set.
    There must, therefore, exist only finitely many factorisations of $G$.
\end{proof}

So, every Left dead end can be written as a (necessarily finite) sum of atoms
in only finitely many ways. But what do these atoms look like? How many of them
are there? We will soon investigate both of these questions, but we first we
make a swift detour in proving a powerful property of $\overline{1}$ that will
yield the only known prime element in our monoid.

\section{In search of a prime}
\label{sec:prime}

We first describe the property of $\overline{1}$ that we are about to prove,
since it is perhaps surprising: given any sum of Left dead ends
$\overline{1}+G_1+\dots+G_n$, it is \emph{always} the best move to play on
$\overline{1}$. That is in every sum of Left dead ends containing
$\overline{1}$, there is exactly one good move.

\begin{lemma}
    \label{lem:1-plus-option}
    If $G$ is a Left dead end, then $\overline{1}+G'\geq G$ for every option
    $G'$ of $G$.
\end{lemma}

\begin{proof}
    We proceed by induction on the birthday of a Left dead end. If $G'=0$, then
    $\overline{1}+G'=\overline{1}\geq G$ since the options of $\overline{1}$
    form a (non-empty) subset of the options of $G$. Assume now that $G'$ is
    non-zero.

    The options of $\overline{1}+G'$ are $G'$ and $\overline{1}+G''$, where
    $G''$ ranges over all options of $G'$. Observe that $G'\geq G'$, and also
    that $\overline{1}+G''\geq G'$ by induction. The result is thus yielded by
    \cref{thm:reduced}.
\end{proof}

\begin{proposition}
    \label{prop:one-option}
    If $G$ is a Left dead end, then $\overline{1}+G=\{\cdot\mid G\}$.
\end{proposition}

\begin{proof}
    Since the options of $\{\cdot\mid G\}$ form a (non-empty) subset of the
    options of $\overline{1}+G$, it is immediate that $\{\cdot\mid
    G\}\geq\overline{1}+G$. By \cref{lem:1-plus-option}, we have
    $\overline{1}+G\geq\{\cdot\mid G\}$, yielding the result.
\end{proof}

It will be useful for us to reframe \cref{prop:one-option} to talk about the
number of good options of a Left dead end $G$, which we do so now.

\begin{proposition}
    \label{prop:one-good-option}
    If $G$ and $G'$ are Left dead ends, then $G$ has exactly one good option
    (equal to) $G'$ if and only if $G=\overline{1}+G'$.
\end{proposition}

\begin{proof}
    First suppose that $G$ has exactly one good option $G'$, then clearly
    $G=\{\cdot\mid G'\}$. By \cref{prop:one-option}, it then follows that
    $G=\overline{1}+G'$.

    Now suppose instead that $G=\overline{1}+G'$. \cref{prop:one-option} then
    tells us that $G=\{\cdot\mid G'\}$, which has exactly one good option, and
    hence it follows from \cref{cor:good-option-replies} that $G$ has exactly
    one good option, which must be equal to $G'$, completing the proof.
\end{proof}

What is surprising is that \cref{prop:one-option} is powerful enough to prove
that $\overline{1}$ is prime---we remind the reader that a game $G$ is prime
if, whenever $G\mid H+K$, it follows that $G\mid H$ or $G\mid K$.

\begin{theorem}
    \label{thm:one-prime}
    The game $\overline{1}$ is prime in the monoid of Left dead ends
    $\mathcal{L}$.
\end{theorem}

\begin{proof}
    Let $H$ and $K$ be Left dead ends. If $\overline{1}\mid H+K$, then there
    exists a Left dead end $G$ such that $\overline{1}+G=H+K$. Now,
    $\overline{1}+G=\{\cdot\mid G\}$ by \cref{prop:one-option}, so it has
    exactly one good option by \cref{prop:one-good-option}. It follows from
    \cref{cor:good-option-replies} that $H+K$ must then also have only one good
    option; say, without loss of generality, some $H'+K$. Thus,
    \begin{align*}
        H+K&=\{\cdot\mid H'+K\}\\
           &=\overline{1}+H'+K
    \end{align*}
    by \cref{prop:one-option}. And so, by cancellativity
    (\cref{thm:pocancellative}), $H=\overline{1}+H'$, which means precisely
    that $\overline{1}\mid H$.
\end{proof}

It follows from \cref{thm:one-prime} that every integer is uniquely
factorisable; we know now that $\overline{n}$ admits a prime factorisation
$n\cdot\overline{1}$ for every $n\geq0$, and such factorisations must be
unique!

It is unclear how we might prove that there exists an atom other than
$\overline{1}$ that is prime. Similarly, all efforts at proving that there
exists an atom that is \emph{not} prime have also failed. We leave this as an
open problem.

\begin{problem}
    Does there exist a prime element other than $\overline{1}$? From another
    angle, can we show that there exists an atom that is not prime?
\end{problem}

We now turn back towards our study of atoms, first introducing a new concept
called \emph{flexibility} that will aid us in our quest.

\section{Flexibility}
\label{sec:flexibility}

The formal birthday of a short game---the height of its game tree---is a
measure of how far away a game is from 0. When playing a Left dead end, once
Right has moved to an integer, he no longer has any \emph{flexibility} in how
to play; he is stuck with a rigid structure of having to play a fixed number of
moves until the component terminates. We have seen in \cref{lem:integer} that a
Left dead end is an integer if and only if every terminal run has the same
length. It seems natural then to consider not the distance of a Left dead end
to 0, but instead its distance to an integer. We define this precisely now.

\begin{definition}
    \label{def:flex}
    If $G$ is a Left dead end, then we define the \emph{flexibility} of $G$,
    written $\flex(G)$, recursively as
    \[
        \flex(G)=1+\max(\flex(G')),
    \]
    where $G'$ ranges over all options of $G$, when $G$ is not an integer, and
    as $\flex(G)\coloneq0$ otherwise.

    If $\flex(G)=k$, then we say $G$ is \emph{$k$-flexible}. If $\flex(G)\leq
    k$, then we say $G$ is \emph{weakly $k$-flexible}.
\end{definition}

As a simple example, the reader may verify that $\flex(W_n)=n-1$ for every
$n\geq1$, and also that $\flex(\{\cdot\mid W_3,\overline{7}\})=3$. We can also
immediately remark that the only 0-flexible atom is $\overline{1}$: by
\cref{cor:race-1-atom}, it is clear that $\overline{1}$ is an atom; and, if
$n>0$, then $\overline{n}=\overline{n-1}+\overline{1}$.

We could have used a construction analogous to how the set of short games is
typically constructed in the following way: define
$\mathcal{F}_0\coloneq\{\overline{n}:n\in\mathbb{N}_0\}$, and then,
recursively,
\[
    \mathcal{F}_{n}\coloneq\mathcal{F}_{n-1}\cup\{\{\cdot\mid
    \mathscr{G}\}:\mathscr{G}\subseteq \mathcal{F}_{n-1}\text{ and
}\mathscr{G}\text{ is finite}\}
\]
for $n>0$; the flexibility of $G$ is then the smallest integer $n$ such that
$G\in\mathcal{F}_n$. This alternative definition may be good to keep in mind,
but it is not how we choose to frame things here.

Those games that are weakly 1-flexible have a useful property: of all the Left
dead ends whose terminal lengths are a superset of some fixed set, there is a
unique, maximal such Left dead end (ordered by `$\geq$'), and it is weakly
1-flexible.

\begin{proposition}
    \label{prop:flex-1-minimal}
    If $G$ and $H$ are Left dead ends with $\flex(G)\leq1$ and
    $\terminal(G)\subseteq\terminal(H)$, then $G\geq H$.
\end{proposition}

\begin{proof}
    Firstly, if $G=0$, then the hypothesis implies $H=0$ by \cref{lem:0-props},
    and so we are done.

    Otherwise, let $G'$ be an arbitrary option of $G$. Now,
    \[
        \terminal(G')\subseteq\terminal(G)-1\subseteq\terminal(H)-1
    \]
    by \cref{lem:terminal-subset} and our hypothesis. Since $\flex(G)\leq1$, it
    follows that $\flex(G')=0$; hence $G'$ must be an integer, and so
    $|\terminal(G')|=1$ by \cref{lem:integer}.  Thus, it follows again by
    \cref{lem:terminal-subset} that there must exist some option $H'$ of $H$
    such that $\terminal(G')\subseteq\terminal(H')$. By induction, $G'\geq H'$,
    and we have the result by \cref{thm:reduced}.
\end{proof}

So, to be explicit, if we consider the set of all Left dead ends whose terminal
lengths are a superset of the set $T=\{t_1,\dots,t_n\}$, then the unique,
maximal such Left dead end is given by
$\{\cdot\mid\overline{t_1},\dots,\overline{t_n}\}$. This is weakly 1-flexible;
it is 1-flexible when $n>1$, and 0-flexible otherwise.

Just as we may sometimes like to consider those options of a game of greatest
birthday, we may also like to consider those options of greatest flexibility.
We say such an option is \emph{versatile}.

\begin{definition}
    If $G$ is a Left dead end, then we say an option $G'$ of $G$ is
    \emph{versatile} if $\flex(G')\geq\flex(G)-1$.
\end{definition}

That is, a versatile option of a game $G$ is simply an option of maximum
flexibility amongst all of its options. It is immediate from \cref{def:flex}
that every non-zero Left dead end has a versatile option. Note also that, if
$G$ is not an integer, then a versatile option necessarily has flexibility
$\flex(G)-1$. But, if $G$ is an integer, then all of its subpositions have
flexibility 0.

It is a simple observation that $\flex(G)<\birth(G)$ for any non-zero Left dead
end $G$ (there is equality in the case of the empty game), which we now prove.

\begin{lemma}
    \label{lem:flex-birth}
    If $G$ is a Left dead end, then $\flex(G)\leq\birth(G)$, with equality if
    and only if $G=0$.
\end{lemma}

\begin{proof}
    It is immediate that $\flex(0)=0=\birth(0)$, and also that
    $\flex(\overline{n})=0<n=\birth(\overline{n})$ for every $n>0$. Proceeding
    via induction on the birthday, suppose $G$ is a non-integer Left dead end
    of birthday $n>1$. Let $G'$ be a non-zero versatile option of $G$; so
    $\birth(G')<\birth(G)$ (recall \cref{cor:comparable-race-birth}). By
    induction, $\flex(G')<\birth(G')$. Since $G'$ is versatile and $G$ is not
    an integer, it is clear that $\flex(G)=\flex(G')+1$. Thus,
    \begin{align*}
        \flex(G)&=\flex(G')+1\\
                &\leq\birth(G')\\
                &<\birth(G),
    \end{align*}
    which completes the proof.
\end{proof}

Just as the terminal lengths of two comparable Left dead ends are comparable,
we obtain a similar result for flexiblity.

\begin{proposition}
    \label{prop:flex-relation}
    If $G$ and $H$ are Left dead ends with $G\geq H$, then
    $\flex(G)\leq\flex(H)$.
\end{proposition}

\begin{proof}
    Suppose first that $\flex(H)=0$; so $H$ is an integer $\overline{n}$, with
    terminal lengths $\terminal(H)=\{n\}$. Since $G\geq H$, it follows from
    \cref{thm:subset-terminal-lengths} that $\terminal(G)=\{n\}$, and so
    $G=\overline{n}$ by \cref{lem:integer}, yielding the result. Assume now
    that $\flex(H)>0$.

    Suppose, for a contradiction, that $\flex(G)>\flex(H)$. We may assume
    $(G,H)$ is a counter-example minimising the birthday of $G+H$. Let $G'$ be
    some versatile option of $G$. By \cref{thm:reduced}, there must exist some
    option $H'$ of $H$ with $G'\geq H'$. Since $\birth(G'+H')<\birth(G+H)$, it
    follows that $\flex(G')\leq\flex(H')$. But $G'$ is versatile, so, by
    supposition,
    \begin{align*}
        \flex(G')&=\flex(G)-1\\
                 &\geq\flex(H)\\
                 &>\flex(H'),
    \end{align*}
    which is the contradiction we were after.
\end{proof}

Speaking practically, this result tells us that when we are looking to remove
dominated options from a Left dead end to arrive at a canonical form, an option
can never be dominated by a less flexible one. It is also immediate that
flexibility is an invariant up to equality, as opposed to just
isomorphism---much like the race and birthday (and terminal lengths) of Left
dead ends.

It can sometimes be useful to view flexibility as the length of a longest run
that doesn't visit an integer except at the end, which we make precise here.

\begin{definition}
    \label{def:flex-run}
    If $G$ is a Left dead end, then we say a run $(G_0,G_1,\dots,G_n)$ of $G$
    of length $n>0$ is \emph{flexible} if $G_i$ is not an integer for $i<n$. We
    define all runs of length 0 to be flexible.

    Furthermore, we say that a flexible run is \emph{maximal} if it ends at an
    integer; that is, if it cannot be extended to a longer flexible run. If a
    flexible run is not maximal, then we say it is \emph{extendable}.
\end{definition}

\begin{lemma}
    \label{lem:flex-is-longest-flex-run}
    If $G$ is a Left dead end, then $\flex(G)$ is equal to the length of a
    longest flexible run of $G$.
\end{lemma}

\begin{proof}
    If $G$ is an integer, then the result is immediate from comparing
    \cref{def:flex,def:flex-run}.

    Suppose now that $G$ is not an integer. In any flexible run
    $(G_0,G_1,\dots,G_n)$ of $G$, it is clear that $\flex(G_i)>\flex(G_{i+1})$
    for every $0\leq i<n$. Thus, the length of a flexible run will be maximised
    precisely by picking successive versatile options; i.e.\ pick $G_i$ to be a
    versatile option of $G_{i-1}$ for every $0<i\leq n$. This will yield a
    flexible run of length $\flex(G)$, and hence also the result.
\end{proof}

It will be convenient to also give names to those options of minimum race and
maximum birthday. We have need only for the latter at this moment, but we
introduce both now given their similarity; and we shall need the former later
on.

\begin{definition}
    \label{def:racing-option}
    If $G$ is a Left dead end, then we define a \emph{racing option} of $G$ to
    be an option $G'$ such that $\race(G')=\race(G)-1$.
\end{definition}

\begin{definition}
    \label{def:stalling-option}
    If $G$ is a Left dead end, then we define a \emph{stalling option} of $G$
    to be an option $G'$ such that $\birth(G')=\birth(G)-1$.
\end{definition}

So, a racing option is an option that moves towards the empty game as quickly
as possible, whereas a stalling option is an option that keeps as far away as
possible from the empty game. It is immediate from the definition of race
(\cref{def:race}) and the definition of birthday (and its equivalence with
formal birthday on the set of Left dead ends via
\cref{cor:comparable-race-birth}) that every non-zero Left dead end has a
racing option and a stalling option (which are not necessarily distinct).

\begin{proposition}
    \label{prop:flex-sums}
    If $G$ and $H$ are non-integer Left dead ends, then
    \begin{enumerate}
        \item
            $\flex(G+\overline{n})=\flex(G)+n$; and
        \item
            $\flex(G+H)=\max({\birth(G)+\flex(H),\flex(G)+\birth(H)})$,
    \end{enumerate}
    for all $n\geq0$.
\end{proposition}

\begin{proof}
    First we prove (1) by induction on $n$. If $n=0$, then the result is
    trivial. Assume now that $n>0$.

    Recall by \cref{prop:one-option} that $G+\overline{n}=\{\cdot\mid
    G+\overline{n-1}\}$. It then follows that
    \begin{align*}
        \flex(G+\overline{n})&=\flex(\{\cdot\mid G+\overline{n-1})
            \quad\text{[\cref{prop:flex-relation}]}\\
        &=1+\flex(G+\overline{n-1}) \quad\text{[\cref{def:flex}]}\\
        &=\flex(G)+n \quad\text{[by induction]}.
    \end{align*}

    We now prove (2). Suppose, without loss of generality, that
    \[
        \birth(G)+\flex(H)\geq\flex(G)+\birth(H),
    \]
    where we now write $M$ for $\birth(G)+\flex(H)$. First we will show that
    $\flex(G+H)\geq M$. Consider a run on $G+H$ given by playing $\birth(G)$
    stalling moves on $G$, followed by $\flex(H)$ versatile moves on $H$. Since
    neither $G$ nor $H$ is an integer, it follows clearly that this is a
    (maximal) flexible run of $G+H$. So, $\flex(G+H)\geq M$ by
    \cref{lem:flex-is-longest-flex-run}.

    Suppose now, for a contradiction, that $\flex(G+H)>M$. Then, again by
    \cref{lem:flex-is-longest-flex-run}, there must exist an extendable
    flexible run of $G+H$ of length $M$. In such a run, at most $\birth(G)$
    moves may ever be played on $G$. If exactly $\birth(G)$ moves were to be
    played on $G$, then exactly $\flex(H)$ moves would be played on $H$, and so
    it is clear that the run would not be extendable. Thus, for some $k>0$,
    there are $\birth(G)-k$ moves played on $G$, and $\flex(H)+k$ moves played
    on $H$; necessarily,
    \begin{equation}
        \label{eq:1}
        \tag{i}
        \flex(H)+k\leq\birth(H).
    \end{equation}
    If $\birth(G)-k\geq\flex(G)$, then it is clear that, again, the run would
    not be extendable. Thus,
    \begin{equation}
        \label{eq:2}
        \tag{ii}
        \birth(G)-k<\flex(G).
    \end{equation}
    But adding \cref{eq:1,eq:2} yields
    \[
        \birth(G)+\flex(H)<\flex(G)+\birth(H),
    \]
    contradicting our earlier supposition, and completing the proof.
\end{proof}

In the proof above, we could have proved the second item by induction as well,
but it doesn't provide much more efficiency, and it is worthwhile to see both
techniques here.

Note that the first item of \cref{prop:flex-sums} could be written as
$\flex(G+\overline{n})=\flex(G)+\birth(\overline{n})$ to look more similar to
the second item. But it is important to realise that they are not the same:
considering the example $\{\cdot\mid0,\overline{2}\}+\overline{1}$, we see that
\begin{align*}
    \flex(\{\cdot\mid0,\overline{2}\}+\overline{1})&=\flex(\{\cdot\mid0,\overline{2}\})+1\\
    &=2,
\end{align*}
but $\birth(\{\cdot\mid0,\overline{2}\})+\flex(\overline{1})=3$. (And hence we
do need the distinction between the two items of \cref{prop:flex-sums}.)

\subsection{Using flexibility to find atoms}

We now seek to investigate the behaviour of Left dead ends of different
flexibilities, and, importantly, find some atoms. Perhaps surprisingly, every
1-flexible Left dead end is an atom, which will be yielded swiftly by the
following proposition.

\begin{proposition}
    \label{prop:flex-birth-ineq}
    If $G,H,$ and $K$ are Left dead ends such that
    \begin{enumerate}
        \item
            $H$ is non-zero;
        \item
            $K$ is not an integer; and
        \item
            $H+K\geq G$;
    \end{enumerate}
    then $\birth(H)<\flex(G)$.
\end{proposition}

\begin{proof}
    Suppose, for a contradiction, that $\birth(H)\geq\flex(G)$. Then, we may
    take a run of $H+K$ of length $\flex(G)$ played entirely on the $H$
    component; say to $H'+K$. By \cref{cor:terminal-comparison}, since $H+K\geq
    G$, there exists a run of $G$ of length $\flex(G)$ such that, in
    particular, the resulting position $G'$ satisfies
    $\terminal(H'+K)\subseteq\terminal(G')$. By definition of $\flex(G)$, the
    position $G'$ must be an integer; so $|\terminal(G')|=1$ by
    \cref{lem:integer}. But $K$ is not an integer by hypothesis, and so
    $|\terminal(H'+K)|\geq|\terminal(K)|>1$ by
    \cref{cor:terminal-length-sum-size} and, again, \cref{lem:integer}. It is
    thus clear that $\terminal(H'+K)\not\subseteq\terminal(G')$, which is a
    contradiction, yielding the result.
\end{proof}

\cref{prop:flex-birth-ineq} is saying something peculiar: if $G$, $H$, and $K$
are Left dead ends (forgetting the condition on $K$ in
\cref{prop:flex-birth-ineq}), where $H+K$ is a molecule and $H+K\geq G$, then,
letting $\mathcal{A}$ denote the set of all atoms that appear in factorisations
of $H+K$, it follows that
\[
    \birth(H+K)<\flex(G)+\min\{\birth(J):\overline{1}\neq J\in\mathcal{A}\}.
\]
Of course, $\min\{\birth(J):\overline{1}\neq J\in\mathcal{A}\}$ could be
arbitrarily large, but it is curious that the least birthday of a non-integer
part of $H+K$, along with the flexibility of $G$, could work together to bound
the birthday of $H+K$ in this way.

We now prove some simple corollaries than can oft be useful in determining
whether games are atoms or not.

\begin{corollary}
    \label{cor:level-molecules}
    If $G$ is a Left dead end with more than one good option, then, for any
    non-zero Left dead ends $H$ and $K$ with $H+K\geq G$, it holds that
    $\birth(H),\birth(K)<\flex(G)$. Additionally, if $H+K=G$, then
    $1<\birth(H),\birth(K)$.
\end{corollary}

\begin{proof}
    Since $G$ has more than one good option, it follows from
    \cref{prop:one-option} that it cannot be split by $\overline{1}$; hence
    neither $H$ nor $K$ is an integer. We then obtain
    $\birth(H),\birth(K)<\flex(G)$ from \cref{prop:flex-birth-ineq}.

    Suppose now that $H+K=G$. If either $H$ or $K$ had birthday 1, then they
    would be equal to $\overline{1}$; hence $G$ would have exactly one good
    option by \cref{prop:one-good-option}, contradicting the hypothesis.
\end{proof}

\begin{corollary}
    \label{cor:big-birth-atom}
    If $G$ is a Left dead end with more than one good option, and
    $\birth(G)>2\cdot\flex(G)-2$, then $G$ is an atom.
\end{corollary}

\begin{proof}
    Suppose, for a contradiction, that $G=H+K$ for some non-zero Left dead ends
    $H$ and $K$. By \cref{cor:level-molecules}, since $G$ has more than one
    good option, it follows that $\birth(H),\birth(K)<\flex(G)$. But
    \[
        2\cdot\flex(G)-2<\birth(G)=\birth(H+K)=\birth(H)+\birth(K)<2\cdot\flex(G)-1
    \]
    by \cref{prop:terminal-length-sums}, which is the contradiction we were
    seeking.
\end{proof}

These results let us begin to classify those $k$-flexible games that are atoms.
We could, of course, extend the following indefinitely, but we just give a few
examples for some low flexibilities that can be helpful in practice when
analysing games.

\begin{proposition}
    If $G$ is a Left dead end, then the following statements are true:
    \begin{itemize}
        \item
            if $\flex(G)=1$, then $G$ is an atom;
        \item
            if $\flex(G)=2$, then either $G$ has exactly one good option (i.e.\
            $\overline{1}\mid G$), or $G$ is an atom;
        \item
            if $\flex(G)=3$, then either $G$ has exactly one good option,
            $G=2\cdot W_2$, or $G$ is an atom.
    \end{itemize}
\end{proposition}

\begin{proof}
    Suppose first that $\flex(G)=1$. Since each option of $G$ is an integer, it
    must follow that $G$ has more than one good move, or else $G$ would be an
    integer, and hence would be 0-flexible. It is then immediate from
    \cref{cor:big-birth-atom} that $G$ must be an atom.

    Suppose now that $\flex(G)=2$, and that $G$ has more than one good option.
    By \cref{lem:flex-birth}, we know that $2<\birth(G)$, and hence $G$ is an
    atom by \cref{cor:big-birth-atom}.

    Finally, suppose that $\flex(G)=3$. If $G=H+K$ has more than one good
    option, where $H$ and $K$ are non-zero Left dead ends, then
    $1<\birth(H),\birth(K)<3$ by \cref{cor:level-molecules}. The only Left dead
    ends of birthday 2 are $\overline{2}$ and $W_2$. If either of $H$ or $K$
    equalled $\overline{2}$, then $G$ would have exactly one good option by
    \cref{prop:one-good-option}, which is a contradiction. Thus, $G=2\cdot
    W_2$, or else $G$ is an atom.
\end{proof}

The fact that all 1-flexible games are atoms will actually also follow
immediately from a stronger result that we prove later (see
\cref{prop:integer-option-atom}).

At the end of \cref{subsec:FFM}, one of the questions we asked was how many
atoms there are in our monoid. Indeed, given our introduction of flexibility,
we may ask a similar question: are there $n$-flexible games for every $n$; what
about $n$-flexible atoms?

Yes: if we take any non-zero, $n$-flexible $G$, then $\{\cdot\mid0,G\}$ is
$(n+1)$-flexible and is an atom by \cref{cor:race-1-atom}; and, for $n>0$,
$\{\cdot\mid G\}=\overline{1}+G$ is an $(n+1)$-flexible molecule. (We needed to
take $n>0$ in the latter statement since, if $G$ is an integer, then
$\overline{1}+G$ is also an integer, and hence it is still 0-flexible.) So, for
every $n>0$, the set of $n$-flexible atoms is (countably) infinite, and, for
every $n\neq1$, the set of $n$-flexible molecules is (countably) infinite, too.
It is worth noting, however, that there are only finitely many $n$-flexible
molecules with more than one good option, for every $n\geq0$.

We are not yet finished with our search for atoms, but we take a brief
interlude in the next section to discuss the lengths of factorisation.

\section{Lengths of factorisations}
\label{sec:factorisation-lengths}

In this short section we provide tools to bound the lengths of factorisations
of Left dead ends. Recall that we write $\longest{G}$ for the length of a
longest factorisation of $G$. We have already seen in \cref{lem:bf} that
$\longest{G}\leq\race(G)$.

Of course, a Left dead end $G$ is an atom if and only if $\longest{G}=1$; it is
a molecule if and only if (it is non-zero and) $\longest{G}\neq1$. We write
$\longest{0}=0$, since its unique factorisation is the empty sum.

\begin{lemma}
    \label{lem:bound}
    If $G$ is a Left dead end, then $\longest{G}\leq\longest{G'}+1$ for every
    good option $G'$ of $G$.
\end{lemma}

\begin{proof}
    Suppose, for a contradiction, that there exists a good option $G'$ of $G$
    and a factorisation $G_1+\dots+G_k$ of $G$ with $k\geq\longest{G'}+2$.
    Since $G'$ is a good option of $G$, there must exist, without loss of
    generality, some option $G_1'$ of $G_1$ such that $H\coloneq
    G_1'+G_2+\dots+G_k=G'$ by \cref{cor:good-option-replies}. Clearly
    $\longest{H}\geq k-1>\longest{G'}$, and hence $H\neq G'$, which is the
    contradiction we were seeking.
\end{proof}

\begin{theorem}
    \label{thm:one-option-bound}
    If $G$ is a Left dead end with exactly one good option $G'$, then
    $\longest{G}=\longest{G'}+1$.
\end{theorem}

\begin{proof}
    By \cref{prop:one-option}, it follows that $G=\overline{1}+G'$. Thus,
    writing a longest factorisation of $G'$ will yield the result immediately
    by \cref{lem:bound}.
\end{proof}

\begin{theorem}
    \label{thm:length-bound}
    If $G$ is a Left dead end with more than one good option, then
    \[
        \longest{G}\leq\min\left\{\race(G),\longest{G'}+1,|\terminal(G)|-1,\frac{\flex(G)+1}{2}\right\},
    \]
    where $G'$ ranges over all good options of $G$.
\end{theorem}

\begin{proof}
    We know that $\longest{G}\leq\min\{\race(G),\longest{G'}+1\}$ from
    \cref{lem:bf,lem:bound}.

    We will write $G_1+\dots+G_k$ for a longest factorisation of $G$. None of
    the $G_i$ is equal to $\overline{1}$, otherwise $G$ would have exactly one
    good option by \cref{prop:one-good-option}, contradicting the hypothesis.

    We show first that $k\leq|\terminal(G)|-1$. Suppose, for a contradiction,
    that $k\geq|\terminal(G)|$. It follows from \cref{lem:integer} that
    $|\terminal(G_i)|\geq2$ for every $i$. By
    \cref{cor:terminal-length-sum-size}, it then follows that
    \begin{align*}
        |\terminal(G_1+\dots+G_k)|&\geq k+1\\
                          &\geq|\terminal(G)|+1,
    \end{align*}
    which is the contradiction we were seeking (it contradicts
    \cref{thm:subset-terminal-lengths}).

    We now show that $2\cdot k\leq\flex(G)+1$. If $G$ is an atom, then $k=1$,
    and so the inequality holds unless perhaps $\flex(G)=0$. But $G$ has more
    than one good option by hypothesis, and hence $G$ is not an integer by
    \cref{prop:one-good-option}; so, $\flex(G)\geq1$ and the inequality holds.

    Assume now that $G$ is a molecule. It follows from
    \cref{cor:level-molecules} that $\birth(G_i)\geq2$ for every $i$. By
    \cref{prop:flex-birth-ineq} and \cref{prop:terminal-length-sums}, we obtain
    \begin{align*}
        \flex(G)&>\birth(G_1+\dots+G_{k-1})\\
                &=\birth(G_1)+\dots+\birth(G_{k-1})\\
                &\geq2k-2,
    \end{align*}
    which yields the result. (Since we can reorder the elements in the
    factorisation however we like, this yields precisely that
    $\flex(G)+1=2\cdot k$ if and only if $G_i=W_2$ for every $i$: clearly
    equality implies $\birth(G_i)=2$ for every $i$, and $W_2$ is the only atom
    of birthday 2; conversely, if $G_i=W_2$ for every $i$, then $\flex(G)=2k-1$
    by \cref{prop:flex-sums}.)
\end{proof}

\cref{thm:one-option-bound,thm:length-bound} work together to provide a bound
that can be recursively computed for any Left dead end. This has been
implemented into the \texttt{gemau} software package \cite{davies:gemau}, and
the reader is invited to explore these results there.

We should also point out that each of these bounds can be strictly best, apart
from $\race(G)$ (but we include it because of its use computationally in the
recursive check). To see why $\race(G)$ cannot be strictly best, consider a
racing option $G'$ of $G$. We know $\race(G')=\race(G)-1$, and so, by
\cref{lem:bf},
\begin{align*}
    \longest{G'}+1&\leq\race(G')+1\\
                &=\race(G).
\end{align*}
So, while $\race(G)$ cannot be strictly best, it can be joint best with
$\longest{G'}+1$.

We now give examples in the table below illustrating the point we just made, as
well as cases where the other bounds are strictly best, and also where the
bounds are not tight; we have emboldened the minimal values of each row.
\begin{center}
    \begin{tabular}{ c | c c c c | c }
        $G$ & $\race(G)$ & $\longest{G'}+1$ & $|\terminal(G)|-1$ &
        $\left\lfloor\frac{\flex(G)+1}{2}\right\rfloor$ & $\longest{G}$
        \\ [1ex]
        \hline
                                         & & & & & \\ [-2.5ex]
        $2\cdot\{\cdot\mid0,\overline{1},\overline{3}\}$ & \textbf{2} &
        \textbf{2} & 5 & 3 & \textbf{2} \\ [0.5ex]
        $\{\cdot\mid0,\overline{3}\}+\{\cdot\mid\overline{1},\overline{2},\overline{3}\}$
                   & 3 & \textbf{2} & 6 & 3 & \textbf{2} \\ [0.5ex]
        $2\cdot\{\cdot\mid\overline{3},\overline{4}\}$ & 8 & 5 & \textbf{2} & 3
                                                       & \textbf{2} \\ [0.5ex]
        $\{\cdot\mid\overline{1},\overline{2},\overline{3}\}$ & 2 & 2 & 2 &
        \textbf{1} & \textbf{1} \\ [0.5ex]
        $2\cdot\{\cdot\mid\overline{1},\overline{2},\overline{3}\}$ & 4 & 3 & 4
                                                                    & 3 &
                                                                    \textbf{2}
    \end{tabular}
\end{center}
All of the above values can be calculated or observed swiftly by hand (or with
computational assistance).

At this point, we've seen some atoms, we've seen a prime element, and we've
seen how to bound the length of factorisations. The question on the tip of the
tongue then is: when is a factorisation unique? This is what we explore in the
next section. Along the way, we will prove some results that yield yet more
atoms for us to enjoy.

\section{Uniqueness of factorisations}
\label{sec:uniqueness}

In this section we prove a number of results related to the uniqueness of
factorisations. We jump in at the deep end with what may appear to be a
complicated proposition, but we will soon see that it is quite simple. We saw
earlier in \cref{lem:bound} that $\longest{G}\leq\longest{G'}+1$ for every good
option $G'$ of $G$; the following result tells us what we can infer when we
have equality.

\begin{proposition}
    \label{prop:good-molecule-options}
    If $G$ is a Left dead end with a good option $H$ such that
    $\longest{G}=\longest{G}+1$, then
    \begin{enumerate}
        \item
            $H$ is a racing option of $G$;
        \item
            there exists a unique atom $K$ such that $G=H+K$, and $K$ is
            terminable;
        \item
            every longest factorisation of $G$ is of the form
            $H_1+\dots+H_k+K$, where $H_1+\dots+H_k$ is a longest factorisation
            of $H$; and
        \item
            if $H$ admits a unique longest factorisation, then so does $G$.
    \end{enumerate}
\end{proposition}

\begin{proof}
    Let $G_1+\dots+G_n$ be an arbitrary longest factorisation of $G$; i.e.\
    such that $n=\longest{G}=\longest{H}+1$. We prove each of the claims in
    order, using the previous ones implicitly in each proof of the next.

    We first prove (1). Since $H$ is a good option of $G$, there must exist,
    without loss of generality, an option $G_1'$ of $G_1$ such that
    \[
        G_1'+G_2+\dots+G_n=H
    \]
    by \cref{cor:good-option-replies}. But $n=\longest{H}+1$, and so it must
    follow that $G_1'=0$, and then $H=G_2+\dots+G_n$. Thus, we obtain
    $G=H+G_1$, where $G_1$ is a terminable atom; so $\race(G)=\race(H)+1$,
    proving (1).

    We now prove (2) and (3). Suppose we write $G=H+K$, where $K$ is an atom.
    Then, by cancellativity (\cref{thm:pocancellative}), we obtain $G_1=K$,
    which simultaneously proves (2) and (3).

    Finally, to prove (4), suppose that $H$ admits a unique longest
    factorisation. This must be $G_2+\dots+G_n$. The result is thus immediate.
\end{proof}

In \cref{prop:good-molecule-options}, if $G$ has exactly one good option $H$,
then $\overline{1}$ is the unique terminable atom $K$ where $G=H+K$; this is
simply \cref{prop:one-good-option}.

The following corollaries can be useful in practice to determine whether games
are atoms, or indeed to construct atoms.

\begin{corollary}
    \label{cor:atom-options}
    If $G$ is a Left dead end with a good option to an atom $H$, then either
    $G$ is an atom, or $H$ is a racing option of $G$ and there exists a unique
    (terminable) atom $K$ such that $G=H+K$; in particular, $G$ is uniquely
    factorisable.
\end{corollary}

\begin{proof}
    By \cref{lem:bound}, we know that $\longest{G}\leq2$. If $\longest{G}=1$,
    then $G$ is an atom. Otherwise $\longest{G}=2$, and the result follows
    immediately from \cref{prop:good-molecule-options}.
\end{proof}

\begin{corollary}
    \label{cor:two-atom-options}
    If $G$ is a Left dead end with good options to distinct atoms $H$ and $K$,
    then either $G$ is an atom, or $H+K$ is the unique factorisation of $G$. In
    particular, if either of $H$ or $K$ is not terminable (i.e.\ if
    $\race(G)\neq2$), then $G$ is an atom.
\end{corollary}

\begin{proof}
    If $G$ is an atom, then we are done. Otherwise, by applying
    \cref{cor:atom-options} twice, it follows that $G=H+J=K+R$, where $J$ and
    $R$ are terminable atoms. But \cref{cor:atom-options} also tells us that
    these factorisations are unique. Since $H$ and $K$ are distinct by
    hypothesis, it must then follow that $G=H+K$ (i.e.\ that $H=R$ and $J=K$).
    In particular, $H$ and $K$ must both be terminable, otherwise we would have
    reached a contradiction, and would conclude that $G$ is an atom.
\end{proof}

\begin{corollary}
    \label{cor:three-atom-options}
    If $G$ is a Left dead end with at least three good options to pairwise
    distinct atoms, then $G$ is an atom.
\end{corollary}

\begin{proof}
    If $G$ is an atom, then we are done. Otherwise, write $G_1$, $G_2$, and
    $G_3$ for three of the atoms that are pairwise distinct good options of
    $G$. By \cref{cor:two-atom-options}, it follows that $G=G_1+G_2=G_2+G_3$.
    By cancellativity (\cref{thm:pocancellative}), we obtain $G_1=G_3$,
    contradicting the hypothesis that $G_1$ and $G_3$ are distinct.
\end{proof}

\cref{cor:atom-options,cor:two-atom-options,cor:three-atom-options} yield an
extraordinary number of atoms (at least at low birthdays), which is discussed
briefly in \cref{app:lattice}.

Fascinatingly, there is also a great deal of similarity between
\cref{cor:atom-options,cor:two-atom-options,cor:three-atom-options} and some of
Siegel's results pertaining to the monoid of impartial games \cite[Proposition
35 and Corollary 36 on pp.~17--18]{siegel:impartial}. See
\cref{app:impartial-similarity} for a further discussion.

\subsection{Good prime-factorisable options}

We are now going to investigate some of the consequences of having good options
that are prime-factorisable. We remind the reader that, since $\overline{1}$ is
the only Left dead end currently known to be prime, there is no currently known
difference in meaning here between the words `prime-factorisable' and
`integer'. Where possible, we will write our results using
`prime-factorisable', which serves as prophylaxis against more primes possibly
being discovered in the future.

\begin{lemma}
    \label{lem:technical}
    If $G\coloneq G_1+\dots+G_n=H_1+\dots+H_m\eqcolon H$ are factorisations of
    a Left dead end $G$, and there is a good option of $G$ playing on the
    component $G_1$ to some prime-factorisable $G_1'$, then there exists an
    index $j$ such that $G_1=H_j$.
\end{lemma}

\begin{proof}
    If $G_1$ is prime, then the result is immediate. Otherwise, by
    cancellativity (\cref{thm:pocancellative}), we may assume that no $G_i$ nor
    $H_j$ is prime. Without loss of generality, there must exist, by
    \cref{cor:good-option-replies}, a good option of $H$ playing on the
    component $H_1$, say to $H_1'$, such that
    \begin{equation}
        \label{eq:temp2}
        G_1'+G_2+\dots+G_n=H_1'+H_2+\dots+H_m.
    \end{equation}
    Since $G_1'$ is prime-factorisable, each atom appearing in its prime
    factorisation divides one of $H_1',H_2,\dots,H_m$. By assumption,
    $H_2,\dots,H_m$ are atoms but not prime, and so it must follow that
    $H_1'=G_1'+R$, where $R$ is some Left dead end. Thus, by cancellativity
    (\cref{thm:pocancellative}), \cref{eq:temp2} becomes
    \[
        G_2+\dots+G_n=R+H_2+\dots+H_m.
    \]
    We then add $G_1$ to both sides and use the fact that $G=H$ in order to
    obtain
    \[
        H_1+\dots+H_m=G_1+R+H_2+\dots+H_m,
    \]
    from which it is clear that $H_1=G_1+R$ by cancellativity
    (\cref{thm:pocancellative}). We know that $H_1$ is an atom, and so
    necessarily $R=0$ (since $G_1\neq0$). Thus, $H_1=G_1$, and we have the
    result.
\end{proof}

We now give one of our main results, which gifts us infinitely many uniquely
factorisable games that we haven't yet seen, and is also practical when
analysing games by hand. We will explore two special cases of the result very
soon afterwards (\cref{thm:racing-moves,thm:stalling-moves}).

\begin{theorem}
    \label{thm:option-has-prime-fact}
    If $G\coloneq G_1+\dots+G_k$ is a factorisation of a Left dead end such
    that, for every $1\leq i\leq k$, there exists a good option of
    $G_i+\dots+G_k$ on the component $G_i$ to some prime-factorisable $G_i'$,
    then $G$ is uniquely factorisable.
\end{theorem}

\begin{proof}
    We will proceed by induction on the length of such a factorisation
    $G_1+\dots+G_k$. If $k=1$, then the factorisation is necessarily unique.
    Now suppose $k>1$. Write $H_1+\dots+H_n$ for an arbitrary factorisation of
    $G$. We will show that $k=n$ and also that $H_i=G_i$ for every $i$.

    By hypothesis, there exists a good option of $G$ moving on $G_1$ to some
    prime-factorisable $G_1'$. \cref{lem:technical} then yields, without loss
    of generality, that $G_1=H_1$. By cancellativity
    (\cref{thm:pocancellative}), we obtain
    \[
        G_2+\dots+G_k=H_2+\dots+H_n,
    \]
    where $G_2+\dots+G_k$ is a factorisation of length $k-1$ that necessarily
    satisfies the hypothesis of the theorem, and hence it is unique by
    induction. Therefore, without loss of generality, it follows that $k=n$ and
    $G_i=H_i$ for every $i$, which is what we set out to prove.
\end{proof}

Before discussing two practical cases of \cref{thm:option-has-prime-fact}, we
first need a quick lemma about where the \emph{good} options in a sum can be.

\begin{lemma}
    \label{lem:good-options-local}
    If $G+H$ is a sum of Left dead ends, and $G'+H$ is
    a good option of $G+H$, then $G'$ is a good option of $G$.
\end{lemma}

\begin{proof}
    Let $G''$ be an option of $G$ satisfying $G'\geq G''$. Then, $G'+H\geq
    G''+H$. Since $G'+H$ is a good option of $G+H$, it follows that
    $G'+H=G''+H$, and hence $G'=G''$ by cancellativity
    (\cref{thm:pocancellative}).
\end{proof}

The converse of \cref{lem:good-options-local} is not true in general: for
example, 0 is a good option of $W_2$, but $\overline{1}$ is not a good option
of $\overline{1}+W_2$. That is, just because a move is good on a component, it
does not mean that the move is good on the sum; but
\cref{lem:good-options-local} tells us that good moves on sums are indeed good
moves on components.

Recall from \cref{def:racing-option,def:stalling-option} that a racing option
of a Left dead end is an option of minimum race, and a stalling option is an
option of maximum birthday. Given a set of atoms $\mathcal{S}$, we will say
that a sum of elements of $\mathcal{S}$ is an
$\mathcal{S}$\emph{-factorisation}.

\begin{theorem}
    \label{thm:racing-moves}
    If $\mathcal{S}\subseteq\mathcal{L}$ is a set of atoms such that every good
    racing option of every $G\in\mathcal{S}$ is prime-factorisable, then every
    $\mathcal{S}$-factorisation is uniquely factorisable.
\end{theorem}

\begin{proof}
    Let $G=G_1+\dots+G_n$ be an $\mathcal{S}$ factorisation. If $G=0$, then the
    conclusion is trivial, so assume $G$ is non-zero. In this case, it follows
    from \cref{thm:subset-terminal-lengths} that there must exist a \emph{good}
    racing option of $G$, say on $G_1$ to $G_1'$ (see
    \cref{lem:good-options-local,prop:terminal-length-sums}). By hypothesis,
    $G_1'$ is prime-factorisable. Thus, \cref{thm:option-has-prime-fact} yields
    the result.
\end{proof}

We can prove a dual result, replacing `racing option' with `stalling option' in
each instance. The proof is identical.

\begin{theorem}
    \label{thm:stalling-moves}
    If $\mathcal{S}\subseteq\mathcal{L}$ is a set of atoms such that every good
    stalling option of every $G\in\mathcal{S}$ is prime-factorisable, then
    every $\mathcal{S}$-factorisation is uniquely factorisable.
\end{theorem}

\cref{thm:racing-moves,thm:stalling-moves} are quite powerful, and they allow
us to say some pretty things. It is important to note that their hypotheses are
not disjoint from one another, but each theorem does currently give some
uniquely factorisable games that the other does not. That is, while
$\overline{1}$ remains the only known prime, the two results are distinct; if
every atom turns out to be prime, then every game would be uniquely
factorisable, in which case both of these results would be irrelevant. We give
some examples now.

As an example of their hypotheses overlapping, we observe that both theorems
yield that every game that can be factorised into 1-flexible atoms is uniquely
factorisable. This includes sums like
$\overline{n}+\{\cdot\mid\overline{1},\overline{4}\}+\{\cdot\mid\overline{7},\overline{8}\}$,
and it is because all of their options (and hence all of their racing and
stalling options) are integers, which must be prime-factorisable since
$\overline{1}$ is prime (\cref{thm:one-prime}).

For \cref{thm:racing-moves} in particular, it is immediate that every sum of
waiting games is uniquely factorisable: so, any game that can be written in the
form $\sum_i^na_i\cdot W_i$ is uniquely factorisable. This is because the only
racing option of a (non-zero) waiting game is 0, which is trivially
prime-factorisable. We cannot currently conclude the same result from
\cref{thm:stalling-moves} since the stalling moves of waiting games (of rank at
least 3) are not known to be prime.

Finally, for \cref{thm:stalling-moves}, we can say that sums like $\{\cdot\mid
W_4,\overline{10}\}+\overline{n}+\{\cdot\mid W_2+W_3, \overline{7}\}$ must be
uniquely factorisable since the (only) stalling option of each atom in that sum
must be an integer, which is prime factorisable since $\overline{1}$ is prime
(\cref{thm:one-prime}). Similarly to before, we cannot currently conclude the
same result from \cref{thm:racing-moves} since the racing move in a game like
$\{\cdot\mid W_4,\overline{10}\}$ is $W_4$, which is not known to be prime.

One will naturally ask whether there exists a way to generalise `stalling
options' and `racing options' in the hypotheses here (of
\cref{thm:racing-moves,thm:stalling-moves}), and indeed there does, but a full
treatment requires the introduction of a significant number of new ideas, and,
crucially, it does not appear to yield anything stronger than what we have
already proved.

\begin{proposition}
    \label{prop:good-prime-fact-option}
    If $G$ is a Left dead end with a good prime-factorisable option $H$, then
    $G$ is uniquely factorisable.
\end{proposition}

\begin{proof}
    If $G$ is an atom, or $H=0$ (which would imply $G$ is an atom by
    \cref{cor:race-1-atom}), then we are done. Otherwise, let $G_1+\dots+G_k$
    be a longest factorisation of $G$, where $k>1$. Since $H$ is a good option
    of $G$, there must exist, without loss of generality, an option $G_1'$ of
    $G_1$ such that
    \[
        G_1'+G_2+\dots+G_k=H
    \]
    by \cref{cor:good-option-replies}. By hypothesis, $H$ is
    prime-factorisable, and hence $G_i$ must be prime for every $i\geq2$. Since
    each of these primes divides $G$, it follows that every factorisation of
    $G$ has length $k$ and is of the form $J+G_2+\dots+G_k$, where $J$ is an
    atom. But, by cancellativity (\cref{thm:pocancellative}), we obtain
    $G_1=J$. Thus, $G$ admits a unique factorisation.
\end{proof}

Interestingly, if we swap the word `prime-factorisable' for `integer' in the
hypothesis of \cref{prop:good-prime-fact-option}, then we obtain a stronger
conclusion---that either $G$ is itself an integer, or else $G$ is an atom---by
exploiting the powerful properties of $\overline{1}$ that are not present in a
generic prime element of a monoid. Of course, since $\overline{1}$ is the only
game known to be prime, it is unclear whether `prime-factorisable' is any more
general a term here than `integer'.

\begin{proposition}
    \label{prop:integer-option-atom}
    If $G$ is a Left dead end with a good option $\overline{n}$, then either
    $G=\overline{n+1}$, or else $G$ is an atom.
\end{proposition}

\begin{proof}
    If $G$ has an option to 0, then $G$ is an atom by \cref{cor:race-1-atom}.
    So, suppose now that $G$ has a good option to $\overline{n}$ with $n>0$.
    Suppose further that $G$ is not an atom; we will show that
    $G=\overline{n+1}$.

    Since $G$ is not an atom, we may write a factorisation $G_1+\dots+G_k$ of
    $G$ for some $k>1$. Since $\overline{n}$ is a good option of $G$, there
    must exist, without loss of generality, an option $G_1'$ of $G_1$ such that
    \[
        G_1'+G_2+\dots+G_k=\overline{n}
    \]
    by \cref{cor:good-option-replies}. Since $\overline{n}$ is
    prime-factorisable (due to \cref{thm:one-prime}), it follows that
    $G_i=\overline{1}$ for every $i\geq2$, and also that
    $G'_1=\overline{n+1-k}$. Our factorisation of $G$ can thus be rewritten as
    $G_1+\overline{k-1}$. But $k>1$, so $G=\{\cdot\mid G_1+\overline{k-2}\}$ by
    \cref{prop:one-option}. Since $G'_1+G_2+\dots+G_k$ is a good option of $G$,
    it now follows from \cref{cor:good-option-replies} that
    \[
        G_1+\overline{k-2}=G'_1+G_2+\dots+G_k=\overline{n}.
    \]
    By cancellativity (\cref{thm:pocancellative}), we obtain
    $G_1=\overline{n+2-k}$. But $G_1$ is an atom, so $G_1=\overline{1}$ and
    $k=n+1$, thus yielding the result.
\end{proof}

\begin{question}
    Can an analogous result to \cref{prop:integer-option-atom} be proved using
    only the general properties of primes in a monoid? Say, if $G$ is a Left
    dead end with a good prime-factorisable option, then must $G$ be either
    prime-factorisable or an atom?
\end{question}

\subsection{Strong atoms}

In the final stretch of this section, we discover swaths of strong atoms. But
first, a simple lemma.

\begin{lemma}
    \label{lem:good-options-n-g}
    If $G$ is a Left dead end, then, for every $n\in\mathbb{N}_1$, $G'$ is a
    good option of $G$ if and only if $G'+(n-1)\cdot G$ is a good option of
    $n\cdot G$.
\end{lemma}

\begin{proof}
    By pocancellativity (\cref{thm:pocancellative}), for every pair of options
    $G'$ and $G''$ of $G$, it follows that $G'\geq G''$ if and only if
    $G'+(n-1)\cdot G\geq G''+(n-1)\cdot G$, yielding the result immediately.
\end{proof}

So, \cref{lem:good-options-n-g} says what we would expect: the good options on
a sum of copies of $G$ correspond precisely to the good options of $G$.

\begin{lemma}
    \label{lem:unique+prime}
    If $G$ and $H$ are Left dead ends, where $G$ is uniquely factorisable, and
    $H$ is prime-factorisable, then $G+H$ is uniquely factorisable.
\end{lemma}

\begin{proof}
    Since $H$ is prime-factorisable, we observe that every factorisation of
    $G+H$ must be of the form $J+H$. Necessarily, $G+H=J+H$, and hence by
    cancellativity (\cref{thm:pocancellative}) we have $G=J$. Since $G$ is
    uniquely factorisable, so is $J$, and we have the result.
\end{proof}

Recall that a strong atom is an atom such that each of its multiples is
uniquely factorisable.

\begin{lemma}
    \label{lem:strong-atom}
    If $G$ is a Left dead end that is an atom with a good prime-factorisable
    option, then $G$ is a strong atom.
\end{lemma}

\begin{proof}
    We need to show that $n\cdot G$ is uniquely factorisable for every
    $n\geq1$. By \cref{lem:good-options-n-g}, it is clear that, for every
    $1\leq i\leq n$, there is a good option of $i\cdot G$ playing on $G$ to
    some prime-factorisable $G'$. \cref{thm:option-has-prime-fact} then yields
    the result.
\end{proof}

\begin{theorem}
    \label{thm:strong}
    If $G$ is a Left dead end with a good prime-factorisable option, then
    $n\cdot G$ is uniquely factorisable for all $n$.
\end{theorem}

\begin{proof}
    If $G$ is itself prime-factorisable, then the result is immediate, and so
    we assume that $G$ is not prime-factorisable. Additionally, if $n=1$, then
    we have the result by \cref{prop:good-prime-fact-option}, so assume that
    $n>1$.

    By hypothesis, we may write, without loss of generality, that
    $G_1'+G_2+\dots+G_t$ is a good prime-factorisable option of $G$; since each
    $G_i$ is an atom, it then follows that $G_i$ must be prime for $i\geq2$.
    Since $G$ is not prime-factorisable by our earlier assumption, we know that
    $G_1$ cannot be prime (but $G_1'$ is prime-factorisable, of course).

    Observe that, since $G_1+\dots+G_t$ is a factorisation of $G$, it must
    follow that $n\cdot G_1+\dots+n\cdot G_t$ is a factorisation of $n\cdot G$.
    Since $G_1'+G_2+\dots+G_t$ is a good option of $G_1+\dots+G_t$, we know by
    \cref{lem:good-options-local} that $G_1'$ is a good option of $G_1$. Since
    $G_1'$ is necessarily prime-factorisable, it then follows from
    \cref{lem:strong-atom} that $G_1$ must be a strong atom. In particular,
    $n\cdot G_1$ is uniquely factorisable; hence, it follows from
    \cref{lem:unique+prime} that $n\cdot G_1+n\cdot(G_2+\dots+G_t)=n\cdot G$ is
    uniquely factorisable.
\end{proof}

It is important to point out, again, that it is not currently known whether
`prime-factorisable' means anything different to `integer'. And, due to
\cref{prop:integer-option-atom}, if $\overline{1}$ turns out to be the only
prime, then \cref{thm:strong} doesn't say anything more than we already knew
from \cref{lem:strong-atom} and \cref{thm:one-prime}; i.e.\ that an atom with a
good integer option is strong, and that integers are uniquely factorisable.

\section{Final remarks}
\label{sec:final-remarks}

In this paper, we have investigated the structure of Left dead ends, and in
particular through the lense of factorisation theory. But the study was not
conclusive, and there remain many open questions to be resolved. The most
important of which, in this author's opinion, is the question of uniqueness: is
every Left dead end uniquely factorisable? That is, is our monoid a unique
factorisation monoid (sometimes also called a \emph{factorial} monoid)?

\begin{problem}
    \label{prob:unique-factoristaion}
    Is every Left dead end uniquely factorisable?
\end{problem}

While we have seen that the Left dead ends have a rigid structure and are
unreasonably well-behaved when compared with some other mis\`ere structures, it
would be quite remarkable indeed if this monoid would be a unique factorisation
monoid. It would imply that the monoid of Left dead ends is a free commutative
monoid (on a countably infinite number of generators); it would be isomorphic
to the monoid natural numbers (indexed from 1) under multiplication, with the
(Left dead end) atoms mapping to the prime numbers.

So, if one were to try and disprove a claim of uniqueness, where should one
look? If one is trying to find a minimal counter-example, then certainly the
game must have more than one good option; if it had exactly one, then the prime
$\overline{1}$ would be a factor, and hence removing the $\overline{1}$ would
yield a smaller counter-example, contradicting minimality. Similarly, a minimal
counter-example must have at least three atoms that are factors; if it had one,
then it would clearly be uniquely factorisable; and if it had two, then it
would follow (from \cref{prop:grothendiek,thm:pocancellative}) that we could
find two factorisations with a common factor, and hence cancelling this factor
would yield a smaller counter-example, contradicting minimality.

Additionally, by \cref{cor:atom-options}, every Left dead end with a good
option to an atom is uniquely factorisable. Thus, a counter-example must not
have a good option to an atom; every good option must be a molecule. And herein
lies the first computational issue: at low birthdays, there are relatively few
molecules (compared to the number of atoms), and so there simply aren't many
opportunities for a counter-example to present itself. But one cannot go to
high birthdays without contending with the combinatorial explosion.Worse still,
we saw in \cref{sec:uniqueness} that large families of games not covered by the
above discussion are also uniquely factorisable. So, perhaps some more
ingenuity than a brute-force search is required.

In terms of the birthday of a minimal counter-example, computations performed
thus far show that one would need to start at those Left dead ends born on day
8. This was an exhaustive check performed using \texttt{gemau}
\cite{davies:gemau}, the details of which can be found in the ancillary files
attached to this work.

Thus, from the above discussion, if the reader seeks an answer to
\cref{prob:unique-factoristaion}, then a minimal counter-example (if it is to
exist) must be a Left dead end of birthday at least 8 that has at least three
factors that are atoms, and whose good options are all molecules. Instead of
brute-force, perhaps randomly generating games, or developing some other ideas,
would prove fruitful.

Finally, besides understanding the uniqueness of factorisations, it would also
be interesting to further investigate the the number of Left dead ends born by
day $n$. We did not discuss this much in this paper, aside from our comments at
the end of \cref{sec:flexibility} regarding the sets of $n$-flexible games, but
the reader is invited to \cref{app:lattice} for a brief overview.

\begin{problem}
    How many atoms and molecules are born by day $n$? More generally, given a
    finite set $S\subseteq\mathbb{N}_1$, how many atoms and molecules $G$ are
    there with $\terminal(G)=S$?
\end{problem}

Our motivation for this paper came from wanting to better understand the
structure of mis\`ere universes. Instead, we have gotten lost in the wonderful
world of Left dead ends. The present author hopes that much more will soon be
discovered. Mis\`ere theory is full of surprises!

\section*{Acknowledgements}

This work was much improved due to a careful reading by Vishal Yadav, as well
as many generous hours of discussion. The author is indebted to Danny Dyer and
Rebecca Milley for many comments; this work would have been completed far later
if it wasn't for their encouragements. The author also wishes to thank Aaron
Siegel for comments early on regarding some of the striking similarities with
results for the impartial mis\`ere monoid, as well as Alfred Geroldinger for
useful and enthusiastic discussion (and for helping the present author
navigate/understand some of the factorisation literature).

\bibliographystyle{plainurl}
\bibliography{bib}

\begin{thebibliography}{10}

\bibitem{allen:investigation}
Meghan~Rose Allen.
\newblock {\em An Investigation of Partizan Mis\`ere Games}.
\newblock Phd thesis, Dalhousie University, Halifax, Canada, July 2009.

\bibitem{angermuller:strong}
Gerhard Angerm\"uller.
\newblock Strong atoms in {K}rull monoids.
\newblock {\em Semigroup Forum}, 101(1):11--18, 2020.
\newblock \href {https://doi.org/10.1007/s00233-019-10035-y} {\path{doi:10.1007/s00233-019-10035-y}}.

\bibitem{blyth:lattices}
T.~S. Blyth.
\newblock {\em Lattices and ordered algebraic structures}.
\newblock Universitext. Springer-Verlag London, Ltd., London, 2005.

\bibitem{calistrate.paulhus.ea:on}
Dan Calistrate, Marc Paulhus, and David Wolfe.
\newblock On the lattice structure of finite games.
\newblock In {\em More games of no chance ({B}erkeley, {CA}, 2000)}, volume~42 of {\em Math. Sci. Res. Inst. Publ.}, pages 25--30. Cambridge Univ. Press, Cambridge, 2002.

\bibitem{davies:gemau}
Alfie Davies.
\newblock \texttt{gemau}, September 2024.
\newblock URL: \url{https://github.com/alfiemd/gemau}.

\bibitem{davies:graff}
Alfie Davies.
\newblock \texttt{graff}, September 2024.
\newblock URL: \url{https://github.com/alfiemd/graff}.

\bibitem{davies:poset}
Alfie Davies.
\newblock \texttt{poset}, September 2024.
\newblock URL: \url{https://github.com/alfiemd/poset}.

\bibitem{davies.mckay.ea:pocancellation}
Alfie~M. Davies, Neil~A. McKay, Rebecca Milley, Richard~J. Nowakowski, and Carlos~P. Santos.
\newblock Pocancellation theorems in {C}ombinatorial {G}ame {T}heory.
\newblock Submitted, 2024.

\bibitem{fisher.mckay.ea:indecomposable}
Michael Fisher, Neil~A. McKay, Rebecca Milley, Richard~J. Nowakowski, and Carlos~P. Santos.
\newblock Indecomposable combinatorial games, 2023.
\newblock URL: \url{https://arxiv.org/abs/2306.07232v1}, \href {https://arxiv.org/abs/2306.07232} {\path{arXiv:2306.07232}}.

\bibitem{fisher.nowakowski.ea:invertible}
Michael Fisher, Richard~J. Nowakowski, and Carlos Pereira~dos Santos.
\newblock Invertible elements of the dicot mis\`ere universe.
\newblock {\em Integers}, 22:Paper No. G6, 11, 2022.

\bibitem{geroldinger:additive}
Alfred Geroldinger.
\newblock Additive group theory and non-unique factorizations.
\newblock In {\em Combinatorial number theory and additive group theory}, Adv. Courses Math. CRM Barcelona, pages 1--86. Birkh\"auser Verlag, Basel, 2009.
\newblock \href {https://doi.org/10.1007/978-3-7643-8962-8} {\path{doi:10.1007/978-3-7643-8962-8}}.

\bibitem{geroldinger.halter-koch:non-unique}
Alfred Geroldinger and Franz Halter-Koch.
\newblock {\em Non-unique factorizations}, volume 278 of {\em Pure and Applied Mathematics (Boca Raton)}.
\newblock Chapman \& Hall/CRC, Boca Raton, FL, 2006.
\newblock Algebraic, combinatorial and analytic theory.
\newblock \href {https://doi.org/10.1201/9781420003208} {\path{doi:10.1201/9781420003208}}.

\bibitem{geroldinger.halter-koch:survey}
Alfred Geroldinger and Franz Halter-Koch.
\newblock Non-unique factorizations: a survey.
\newblock In {\em Multiplicative ideal theory in commutative algebra}, pages 207--226. Springer, New York, 2006.
\newblock URL: \url{https://doi.org/10.1007/978-0-387-36717-0_13}, \href {https://doi.org/10.1007/978-0-387-36717-0\_13} {\path{doi:10.1007/978-0-387-36717-0\_13}}.

\bibitem{geroldinger.zhong:factorization}
Alfred Geroldinger and Qinghai Zhong.
\newblock Factorization theory in commutative monoids.
\newblock {\em Semigroup Forum}, 100(1):22--51, 2020.
\newblock \href {https://doi.org/10.1007/s00233-019-10079-0} {\path{doi:10.1007/s00233-019-10079-0}}.

\bibitem{larsson.nowakowski.ea:infinitely}
U.~Larsson, R.~J. Nowakowski, and C.~P. Santos.
\newblock Infinitely many absolute universes, 2023.
\newblock URL: \url{https://arxiv.org/abs/2303.05198v1}, \href {https://arxiv.org/abs/2303.05198} {\path{arXiv:2303.05198}}.

\bibitem{larsson.milley.ea:recursive}
Urban Larsson, Rebecca Milley, Richard Nowakowski, Gabriel Renault, and Carlos Santos.
\newblock Recursive comparison tests for dicot and dead-ending games under {M }is\`ere play.
\newblock {\em Integers}, 21B:Paper No. A16, 14, 2021.

\bibitem{larsson.nowakowski.ea:absolute}
Urban Larsson, Richard~J. Nowakowski, and Carlos~P. Santos.
\newblock Absolute combinatorial game theory, 2021.
\newblock URL: \url{https://arxiv.org/abs/1606.01975v3}, \href {https://arxiv.org/abs/1606.01975} {\path{arXiv:1606.01975}}.

\bibitem{mesdal.ottaway:simplification}
G.~A. Mesdal and P.~Ottaway.
\newblock Simplification of partizan games in mis\`ere play.
\newblock {\em Integers}, 7:G06, 12, 2007.

\bibitem{milley:restricted}
Rebecca Milley.
\newblock {\em Restricted Universes of Partizan Mis\`ere Games}.
\newblock Phd thesis, Dalhousie University, Halifax, Canada, March 2013.

\bibitem{milley.renault:dead}
Rebecca Milley and Gabriel Renault.
\newblock Dead ends in mis\`ere play: the mis\`ere monoid of canonical numbers.
\newblock {\em Discrete Math.}, 313(20):2223--2231, 2013.
\newblock \href {https://doi.org/10.1016/j.disc.2013.05.023} {\path{doi:10.1016/j.disc.2013.05.023}}.

\bibitem{milley.renault:restricted}
Rebecca Milley and Gabriel Renault.
\newblock Restricted developments in partizan mis\`ere game theory.
\newblock In {\em Games of no chance 5}, volume~70 of {\em Math. Sci. Res. Inst. Publ.}, pages 113--123. Cambridge Univ. Press, Cambridge, 2019.

\bibitem{milley.renault:invertible}
Rebecca Milley and Gabriel Renault.
\newblock The invertible elements of the monoid of dead-ending mis\`ere games.
\newblock {\em Discrete Math.}, 345(12):Paper No. 113084, 13, 2022.
\newblock \href {https://doi.org/10.1016/j.disc.2022.113084} {\path{doi:10.1016/j.disc.2022.113084}}.

\bibitem{nathanson:additive}
Melvyn~B. Nathanson.
\newblock {\em Additive number theory}, volume 165 of {\em Graduate Texts in Mathematics}.
\newblock Springer-Verlag, New York, 1996.
\newblock Inverse problems and the geometry of sumsets.
\newblock \href {https://doi.org/10.1007/978-1-4757-3845-2} {\path{doi:10.1007/978-1-4757-3845-2}}.

\bibitem{plambeck:taming}
Thane~E. Plambeck.
\newblock Taming the wild in impartial combinatorial games.
\newblock {\em Integers}, 5(1):G5, 36, 2005.

\bibitem{plambeck.siegel:misere}
Thane~E. Plambeck and Aaron~N. Siegel.
\newblock Mis\`ere quotients for impartial games.
\newblock {\em J. Combin. Theory Ser. A}, 115(4):593--622, 2008.
\newblock \href {https://doi.org/10.1016/j.jcta.2007.07.008} {\path{doi:10.1016/j.jcta.2007.07.008}}.

\bibitem{siegel:cgsuite}
Aaron~N. Siegel.
\newblock \texttt{CGSuite}.
\newblock \url{https://www.cgsuite.org/}, 2003--2024.

\bibitem{siegel:combinatorial}
Aaron~N. Siegel.
\newblock {\em Combinatorial game theory}, volume 146 of {\em Graduate Studies in Mathematics}.
\newblock American Mathematical Society, Providence, RI, 2013.
\newblock \href {https://doi.org/10.1090/gsm/146} {\path{doi:10.1090/gsm/146}}.

\bibitem{siegel:misere}
Aaron~N. Siegel.
\newblock Mis\`ere canonical forms of partizan games.
\newblock In {\em Games of no chance 4}, volume~63 of {\em Math. Sci. Res. Inst. Publ.}, pages 225--239. Cambridge Univ. Press, New York, 2015.

\bibitem{siegel:impartial}
Aaron~N. Siegel.
\newblock On the structure of mis\`ere impartial games.
\newblock {\em Integers}, 21B:Paper No. A20, 26, 2021.

\bibitem{siegel:on}
Aaron~N. Siegel.
\newblock On the general dead-ending universe of partizan games, 2023.
\newblock URL: \url{https://arxiv.org/abs/2312.16259v1}, \href {https://arxiv.org/abs/2312.16259} {\path{arXiv:2312.16259}}.

\end{thebibliography}

\appendix

\section{Similarities with the impartial monoid}
\label{app:impartial-similarity}

It was pointed out to the current author by Aaron Siegel (in a private
communication) that there is a striking similarity between some of our results
here for Left dead ends and the results on the monoid of impartial mis\`ere
forms. This similarity later transpired to be even greater than it appeared at
first glance. See \cite[\S6--7 on pp.~15--20]{siegel:impartial} in particular,
which contains most of the commonalities.

The results may not immediately appear as similar as they actually are because
the language of \cite{siegel:impartial} is not congruent with the langugage we
have used here---the language more commonly found in the algebra literature.
For example, Siegel's `prime' is our `atom', and `prime partition' is our
`factorisation'. It would be helpful to standardise this, so that comparisons
can be made more easily. Perhaps some of our results here could also be used to
extend the results for the impartial monoid. Further investigation is highly
warranted.

It is important to note that it is not just the theorem \emph{statements} that
are almost identical, but the \emph{proofs}, too. (We should point out,
however, that these really are distinct monoids; the monoid of Left dead ends
is reduced, while the impartial monoid has a torsion element, for example.) The
question then is why are these theorems and proofs so similar? This can partly
be attributed to how simple the two monoids are, and in particular to how they
are both `one-sided': each impartial game can be identified with its set of
options (there is no need for a distinction between Left and Right options,
since they are always the same); similarly, each Left dead end can be
identified with its Right options (since there are never any Left options to
involve in our arguments).

But perhaps it has more to do with the fact that these are not your typical
monoids; these are monoids of combinatorial games, whose elements are defined
in a strange, recursive manner (by their Left and Right options). Is it
possible that the techniques can work for more games than just impartial games
and Left dead ends? That is, is it possible that a more general factorisation
theory of mis\`ere monoids can be developed? After all, this study of Left dead
ends, along with Siegel's study of the impartial monoid, are so far the only
two efforts at a mis\`ere factorisation theory.

It is notable, however, that there was a recent investigation into a
factorisation theory for normal play games by Fisher, McKay, Milley,
Nowakowski, and Santos \cite{fisher.mckay.ea:indecomposable}. Of course, the
set of values in normal play forms a group, and so the typical factorisation
theory of monoids doesn't apply (since everything is a unit). As such, they had
to \emph{force} the factorisations to work by restricting the birthdays of the
factors, which is natural. Our definition of \emph{molecule} is effectively
their \emph{strongly decomposable}, which the reader can compare
\cite[Definition 2 on p.~3]{fisher.mckay.ea:indecomposable}.

We also remark that there is a natural bijection between the monoid of Left
dead ends and the monoid of impartial forms: given an impartial form, simply
remove every Left option from every subposition; what remains must be a Left
dead end. Conversely, given a Left dead end, one can \emph{mirror} all of the
right options, starting from the bottom of the game tree and working up, to
necessarily obtain an impartial form. This correspondence, however, appears to
be superficial: the two monoids are fundamentally different; the monoid of Left
dead ends is partially ordered, whereas there is no way to have $G>H$ for
impartial forms $G$ and $H$ (they are either equal or incomparable).

We conclude with a short subsection lifting a result from Siegel's work that
the present author had not thought of for the case of Left dead ends, which
yields a reasonably efficient algorithm for finding the factors of a Left dead
end.

\subsection{Factoring algorithm}

All we need to do here is repurpose Siegel's Difference Lemma \cite[Lemma 22 on
p.~14]{siegel:impartial} for Left dead ends. The statement here is arguably
slightly simpler.

\begin{lemma}
    \label{lem:difference}
    If $G$ and $H$ are Left dead ends satifying $H\mid G$, then either:
    \begin{enumerate}
        \item
            there exist good options $G'$ and $H'$ such that $G+H'=G'+H$; or
            else
        \item
            for each good option $G_i'$ (of $G$), there exists a Left dead end
            $X_i$ such that $G_i'=H+X_i$, and $G=H+\{\cdot\mid X_i\}$ (where
            $X_i$ here ranges over all $i$).
    \end{enumerate}
\end{lemma}

\begin{proof}
    If $G=0$, then the result is immediate, so assume otherwise. We may also
    assume that there exist no good options $G'$ and $H'$ such that
    $G+H'=G'+H$. This implies that $G\neq H$.

    Now, since $H\mid G$ and $G\neq H$, there must exist some non-zero Left
    dead end $K$ such that $G=H+K$. By \cref{cor:good-option-replies}, for each
    good option $G_i'$, either there exists some (good) option $H_i'$ with
    $G_i'=H_i'+K$, or else there exists some (good) option $K_i'$ with
    $G_i'=H+K_i'$. If we are ever in the first case, with $G_i'=H_i'+K$, then
    adding $H$ to both sides yields $G_i'+H=H_i'+G$, contradicting our initial
    assumption. Thus, we must always be in the second case where $G_i'=H+K_i'$,
    from which it follows (from \cref{thm:reduced}) that $G\geq H+\{\cdot\mid
    K_i'\}$ (where $K_i'$ ranges over all $i$). Now, clearly $\{\cdot\mid
    K_i'\}\geq K$, and hence
    \begin{align*}
        H+\{\cdot\mid K_i'\}&\geq H+K\\
                            &=G,
    \end{align*}
    from which we obtain $G=H+\{\cdot\mid K_i'\}$, which is the result.
\end{proof}

Then, as Siegel notes for the impartial monoid, we too now have a reasonable
algorithm for finding the factors of an arbitrary Left dead end. Of course, we
could always use a brute-force algorithm to find all factors (by generating all
games born before the game we are factoring, which must be finite, and then
trying every possible combination), but the algorithm yielded for us here is
much more practical. We can follow Siegel's instructions almost exactly, but we
write it out here explicitly. The idea is that we are going to recursively
compute all of the factors of each good option $G'$ of $G$, and then
reconstruct the possible factors of $G$ using \cref{lem:difference}.

Let $G$ be a Left dead end. If $G$ is an atom, then we already have all of the
factors of $G$ (they are 0 and $G$ itself). Otherwise, we may write $G=H+K$ for
some non-zero Left dead ends $H$ and $K$. We need to show that we can reover
$H$ and $K$. Obviously, we have $H\mid G$ and $K\mid G$, and so we may apply
\cref{lem:difference}.

Suppose first that both $H$ and $K$ satisfy the first item in
\cref{lem:difference}. In particular, for $H$, we see that $G+H'=G'+H$. Adding
$K$ to both sides yields $G+H'+K=G'+G$. Then, by cancellativity
(\cref{thm:pocancellative}), we obtain $H'+K=G'$. Thus, we can recover $K$ as a
factor of some option $G'$ of $G$. We can apply an identical argument in order
to recover $H$ as a factor of some (perhaps different) option of $G$.

Now suppose instead, without loss of generality, that $H$ does not satisfy the
first item in \cref{lem:difference}. Then $H$ must satisfy the second item: for
each good option $G_i'$ of $G$, there exists some Left end $X_i$ such that
$G_i'=H+X_i$, and $G=H+\{\cdot\mid X_i\}$. So, we can recover $H$ and each
$X_i$ from the factors of the $G_i'$, and we can also then recover $K$; we have
$H+K=G=H+\{\cdot\mid X_i\}$, and so $K=\{\cdot\mid X_i\}$ by cancellativity
(\cref{thm:pocancellative}).

Thus, a shortcut to factorising a Left dead end is to compute all of the
factors of its good options recursively, and then reconstruct the possible
factors of the Left dead end according to the strategy above.

\section{}
\label{app:formal-race}

When we have a short game, in normal play or in mis\'ere play, and we apply
some reduction to it (removing a dominated option, or bypassing a reversible
one), it is impossible to strictly increase the formal birthday. It thus
follows that the birthday of a game is exactly the formal birthday of the
game's canonical form.

Unfortunately, it is not so simple for the race of a game. Consider, in normal
play, the form $G=\{0,1\mid\cdot\}$. Clearly, the formal race of this game is
1, since Left can terminate the game immediately. But, in normal play,
$G=\{1\mid\cdot\}$, which clearly has formal race 2. Since $\{1\mid\cdot\}$ is
the canonical form of $G$, we have seen that canonical forms may have higher
formal race than other games in the equivalence class. Must it be the highest?
Alas, it need not be, for we also have the equality $G=\{3/2\mid\cdot\}$, which
has formal race 3. Thus, the formal race of a game's canonical form could be
higher or lower than the formal race of the game.

As such, instead of creating a split for race and formal race, we have chosen
to discuss only a single concept, which is functionally equivalent to formal
race. This is the same motivation for how we defined the terminal lengths of a
game, without discussing \emph{formal} terminal lengths. Perhaps creating a
distinction could lead to some merit elsewhere, but that merit would not be
realised in this study of Left dead ends.

\section{}
\label{app:alt-proof}

While the inductive proof of \cref{thm:subset-terminal-lengths} is easy on the
eyes, we provide here an alternative proof that constructs a distinguishing
game. The intuition is simply the observation that a player would sometimes
like the option to pass when playing a sum of combinatorial games. We can
extend this idea and construct a game $C_n$ in which Right would want to pass
exactly $n$ times before having to play again. Thus, if $G$ is a Left dead end
with $n\in\terminal(G)$, and $H$ is a Left dead end with
$n\not\in\terminal(H)$, then Right will win playing first on $G+C_n$, since he
can pass $n$ times, but he cannot do the same on $H+C_n$. 

\begin{namedthm*}{\cref{thm:subset-terminal-lengths}}
    If $G$ and $H$ are Left dead ends with $G\geq H$, then
    $\terminal(G)\subseteq\terminal(H)$.
\end{namedthm*}

\begin{proof}
    Suppose, for a contradiction, that $G\geq H$ and $\terminal(G)\not\subseteq
    \terminal(H)$. There must then exist some
    $n\in\terminal(G)\backslash\terminal(H)$, and so, by \cref{thm:reduced} and
    \cref{lem:0-props}, it follows that both $G$ and $H$ are non-zero. Now
    construct a game (to be our distinguishing game) $C_0\coloneq\{0,*\mid*\}$,
    and then, recursively,
    \[
        C_k\coloneq\{C_{k-1}\mid0\}
    \]
    for $k>0$. It is clear that
    \[
        \outcomeR(G+C_n)=\mathscr{R}<\mathscr{L}=\outcomeR(H+C_n),
    \]
    and so $G\not\geq H$, yielding the desired contradiction.
\end{proof}

Note that our distinguishing game (in the proof above) was in the dicot
universe $\mathcal{D}$. Since the partial order of Left dead ends is
multiversal (i.e.\ given \cref{cor:multiversal}), it must be the case that two
Left dead ends are distinguishable if and only if they are distinguishable
modulo $\mathcal{D}$ (recall that $\mathcal{D}$ is a subuniverse of
\emph{every} universe).

\section{Counting Left dead ends}
\label{app:lattice}

Siegel has already computed the number of Left dead ends born by day $n$ for
each $n\leq5$ \cite[p.~9]{siegel:on}. While we do not attempt to extend this
particular count here (owing to the combinatorial explosion), we do add data
corresponding to the number of atoms and molecules born by day $n$. In addition
to counting the number of molecules born by day $n$, we also count what we call
here the \emph{non-trivial} molecules born by day $n$; these are simply
molecules that are not divisible by $\overline{1}$.

In the table below, note that all of the entries in the first row were first
due to Siegel (although we have also verified the counts).
\begin{center}
    \begin{tabular}{ l | c c c c c c c c }
        \ldots born by day \ldots & 0 & 1 & 2 & 3 & 4 & 5 & 6 & 7 \\
        \hline
        Left dead ends & 1 & 2 & 4 & 10 & 52 & 21278 & ? & ? \\
        atoms & 0 & 1 & 2 & 6 & 41 & 21221 & ? & ? \\
        molecules & 0 & 0 & 1 & 3 & 10 & 56 & 21328 & ? \\
        non-trivial molecules & 0 & 0 & 0 & 0 & 1 & 5 & 51 & 21375 \\
    \end{tabular}
\end{center}

When we mentioned in \cref{sec:uniqueness} that
\cref{cor:atom-options,cor:two-atom-options,cor:three-atom-options} yield an
extraordinary number of atoms (at least at low birthdays), it was precisely
because of this table. At these low birthdays, the number of atoms dwarfs the
number of molecules, and since having a good option to an atom makes it more
likely that the game is itself an atom, it seems reasonable to think that this
pattern will continue; that more atoms will beget more and more atoms.

We include here in \cref{fig:hasse4} the Hasse diagram of the (non-zero) Left
dead ends born by day 4, which was generated using \texttt{poset} and
\texttt{graff} \cite{davies:poset,davies:graff}. The Hasse diagrams of the
(non-zero) Left dead ends born by days 2 and 3 can be found in \cite[Figure 1
on p.~9]{siegel:on}. We also computed the Hasse diagram of the (non-zero) Left
dead ends born by day 5, but the reader will surely agree that the margin is
too narrow to contain 21277 vertices. Instead, the brave reader can look at the
ancillary files accompanying this work. It would be interesting to further
explore the properties of these lattices.

\begin{figure}
    \centering
    \begin{tikzpicture}[>=latex',line join=bevel,every node/.style={scale=0.6}]
        \node (9) at (23.8bp,279.0bp) [draw,circle, fill] {};
        \node (5) at (1.8bp,239.4bp) [draw,circle, fill] {};
        \node (18) at (23.8bp,239.4bp) [draw,circle, fill] {};
        \node (23) at (45.8bp,239.4bp) [draw,circle, fill] {};
        \node (14) at (1.8bp,199.8bp) [draw,circle, fill] {};
        \node (22) at (23.8bp,199.8bp) [draw,circle, fill] {};
        \node (20) at (89.8bp,318.6bp) [draw,circle, fill] {};
        \node (6) at (111.8bp,279.0bp) [draw,circle, fill] {};
        \node (33) at (89.8bp,279.0bp) [draw,circle, fill] {};
        \node (46) at (67.8bp,279.0bp) [draw,circle, fill] {};
        \node (15) at (111.8bp,239.4bp) [draw,circle, fill] {};
        \node (26) at (89.8bp,239.4bp) [draw,circle, fill] {};
        \node (21) at (133.8bp,318.6bp) [draw,circle, fill] {};
        \node (7) at (155.8bp,279.0bp) [draw,circle, fill] {};
        \node (34) at (133.8bp,279.0bp) [draw,circle, fill] {};
        \node (16) at (155.8bp,239.4bp) [draw,circle, fill] {};
        \node (25) at (133.8bp,239.4bp) [draw,circle, fill] {};
        \node (29) at (78.8bp,81.0bp) [draw,circle, fill] {};
        \node (8) at (89.8bp,41.4bp) [draw,circle, fill] {};
        \node (42) at (67.8bp,41.4bp) [draw,circle, fill] {};
        \node (17) at (78.8bp,1.8bp) [draw,circle, fill] {};
        \node (24) at (45.8bp,318.6bp) [draw,circle, fill] {};
        \node (37) at (45.8bp,279.0bp) [draw,circle, fill] {};
        \node (48) at (111.8bp,160.2bp) [draw,circle, fill] {};
        \node (10) at (111.8bp,120.6bp) [draw,circle, fill] {};
        \node (51) at (89.8bp,120.6bp) [draw,circle, fill] {};
        \node (47) at (67.8bp,120.6bp) [draw,circle, fill] {};
        \node (19) at (100.8bp,81.0bp) [draw,circle, fill] {};
        \node (1) at (111.8bp,358.2bp) [draw,circle, fill] {};
        \node (11) at (155.8bp,318.6bp) [draw,circle, fill] {};
        \node (12) at (67.8bp,318.6bp) [draw,circle, fill] {};
        \node (13) at (111.8bp,318.6bp) [draw,circle, fill] {};
        \node (2) at (133.8bp,358.2bp) [draw,circle, fill] {};
        \node (3) at (67.8bp,358.2bp) [draw,circle, fill] {};
        \node (4) at (89.8bp,358.2bp) [draw,circle, fill] {};
        \node (35) at (1.8bp,160.2bp) [draw,circle, fill] {};
        \node (36) at (45.8bp,199.8bp) [draw,circle, fill] {};
        \node (39) at (111.8bp,199.8bp) [draw,circle, fill] {};
        \node (49) at (67.8bp,239.4bp) [draw,circle, fill] {};
        \node (38) at (133.8bp,199.8bp) [draw,circle, fill] {};
        \node (50) at (56.8bp,81.0bp) [draw,circle, fill] {};
        \node (27) at (23.8bp,160.2bp) [draw,circle, fill] {};
        \node (28) at (45.8bp,160.2bp) [draw,circle, fill] {};
        \node (31) at (67.8bp,199.8bp) [draw,circle, fill] {};
        \node (32) at (89.8bp,199.8bp) [draw,circle, fill] {};
        \node (30) at (155.8bp,199.8bp) [draw,circle, fill] {};
        \node (40) at (23.8bp,120.6bp) [draw,circle, fill] {};
        \node (44) at (67.8bp,160.2bp) [draw,circle, fill] {};
        \node (41) at (45.8bp,120.6bp) [draw,circle, fill] {};
        \node (45) at (89.8bp,160.2bp) [draw,circle, fill] {};
        \node (43) at (133.8bp,160.2bp) [draw,circle, fill] {};
        \draw [] (9) ..controls (20.186bp,271.82bp) and (5.0588bp,245.97bp)  ..
            (5);
        \draw [] (9) ..controls (23.8bp,271.04bp) and (23.8bp,247.49bp)  ..
            (18);
        \draw [] (9) ..controls (27.414bp,271.82bp) and (42.541bp,245.97bp)  ..
            (23);
        \draw [] (5) ..controls (1.8bp,231.44bp) and (1.8bp,207.89bp)  .. (14);
        \draw [] (5) ..controls (5.414bp,232.22bp) and (20.541bp,206.37bp)  ..
            (22);
        \draw [] (20) ..controls (93.414bp,311.42bp) and (108.54bp,285.57bp)
            .. (6);
        \draw [] (20) ..controls (89.8bp,310.64bp) and (89.8bp,287.09bp)  ..
            (33);
        \draw [] (20) ..controls (86.186bp,311.42bp) and (71.059bp,285.57bp)
            .. (46);
        \draw [] (6) ..controls (111.8bp,271.04bp) and (111.8bp,247.49bp)  ..
            (15);
        \draw [] (6) ..controls (108.19bp,271.82bp) and (93.059bp,245.97bp)  ..
            (26);
        \draw [] (21) ..controls (137.41bp,311.42bp) and (152.54bp,285.57bp)
            .. (7);
        \draw [] (21) ..controls (133.8bp,310.64bp) and (133.8bp,287.09bp)  ..
            (34);
        \draw [] (21) ..controls (124.71bp,312.42bp) and (77.901bp,285.75bp)
            .. (46);
        \draw [] (7) ..controls (155.8bp,271.04bp) and (155.8bp,247.49bp)  ..
            (16);
        \draw [] (7) ..controls (152.19bp,271.82bp) and (137.06bp,245.97bp)  ..
            (25);
        \draw [] (29) ..controls (80.896bp,72.836bp) and (88.224bp,47.787bp)
            .. (8);
        \draw [] (29) ..controls (76.704bp,72.836bp) and (69.376bp,47.787bp)
            .. (42);
        \draw [] (8) ..controls (87.704bp,33.236bp) and (80.376bp,8.1874bp)  ..
            (17);
        \draw [] (24) ..controls (42.186bp,311.42bp) and (27.059bp,285.57bp)
            .. (9);
        \draw [] (24) ..controls (45.8bp,310.64bp) and (45.8bp,287.09bp)  ..
            (37);
        \draw [] (24) ..controls (49.414bp,311.42bp) and (64.541bp,285.57bp)
            .. (46);
        \draw [] (48) ..controls (111.8bp,152.24bp) and (111.8bp,128.69bp)  ..
            (10);
        \draw [] (48) ..controls (108.19bp,153.02bp) and (93.059bp,127.17bp)
            .. (51);
        \draw [] (48) ..controls (104.57bp,153.02bp) and (74.318bp,127.17bp)
            .. (47);
        \draw [] (10) ..controls (106.38bp,113.42bp) and (83.688bp,87.57bp)  ..
            (29);
        \draw [] (10) ..controls (109.7bp,112.44bp) and (102.38bp,87.387bp)  ..
            (19);
        \draw [] (1) ..controls (119.03bp,351.02bp) and (149.28bp,325.17bp)  ..
            (11);
        \draw [] (1) ..controls (104.57bp,351.02bp) and (74.318bp,325.17bp)  ..
            (12);
        \draw [] (1) ..controls (111.8bp,350.24bp) and (111.8bp,326.69bp)  ..
            (13);
        \draw [] (11) ..controls (146.71bp,312.42bp) and (99.901bp,285.75bp)
            .. (33);
        \draw [] (11) ..controls (152.19bp,311.42bp) and (137.06bp,285.57bp)
            .. (34);
        \draw [] (2) ..controls (126.57bp,351.02bp) and (96.318bp,325.17bp)  ..
            (20);
        \draw [] (2) ..controls (133.8bp,350.24bp) and (133.8bp,326.69bp)  ..
            (21);
        \draw [] (2) ..controls (137.41bp,351.02bp) and (152.54bp,325.17bp)  ..
            (11);
        \draw [] (12) ..controls (71.414bp,311.42bp) and (86.541bp,285.57bp)
            .. (33);
        \draw [] (12) ..controls (64.186bp,311.42bp) and (49.059bp,285.57bp)
            .. (37);
        \draw [] (3) ..controls (71.414bp,351.02bp) and (86.541bp,325.17bp)  ..
            (20);
        \draw [] (3) ..controls (64.186bp,351.02bp) and (49.059bp,325.17bp)  ..
            (24);
        \draw [] (3) ..controls (67.8bp,350.24bp) and (67.8bp,326.69bp)  ..
            (12);
        \draw [] (13) ..controls (115.41bp,311.42bp) and (130.54bp,285.57bp)
            .. (34);
        \draw [] (13) ..controls (102.71bp,312.42bp) and (55.901bp,285.75bp)
            .. (37);
        \draw [] (4) ..controls (97.028bp,351.02bp) and (127.28bp,325.17bp)  ..
            (21);
        \draw [] (4) ..controls (82.572bp,351.02bp) and (52.318bp,325.17bp)  ..
            (24);
        \draw [] (4) ..controls (93.414bp,351.02bp) and (108.54bp,325.17bp)  ..
            (13);
        \draw [] (14) ..controls (1.8bp,191.84bp) and (1.8bp,168.29bp)  ..
            (35);
        \draw [] (18) ..controls (20.186bp,232.22bp) and (5.0588bp,206.37bp)
            .. (14);
        \draw [] (18) ..controls (27.414bp,232.22bp) and (42.541bp,206.37bp)
            .. (36);
        \draw [] (15) ..controls (111.8bp,231.44bp) and (111.8bp,207.89bp)  ..
            (39);
        \draw [] (33) ..controls (93.414bp,271.82bp) and (108.54bp,245.97bp)
            .. (15);
        \draw [] (33) ..controls (86.186bp,271.82bp) and (71.059bp,245.97bp)
            .. (49);
        \draw [] (16) ..controls (152.19bp,232.22bp) and (137.06bp,206.37bp)
            .. (38);
        \draw [] (34) ..controls (137.41bp,271.82bp) and (152.54bp,245.97bp)
            .. (16);
        \draw [] (34) ..controls (124.71bp,272.82bp) and (77.901bp,246.15bp)
            .. (49);
        \draw [] (42) ..controls (69.896bp,33.236bp) and (77.224bp,8.1874bp)
            .. (17);
        \draw [] (37) ..controls (42.186bp,271.82bp) and (27.059bp,245.97bp)
            .. (18);
        \draw [] (37) ..controls (49.414bp,271.82bp) and (64.541bp,245.97bp)
            .. (49);
        \draw [] (19) ..controls (95.379bp,73.823bp) and (72.688bp,47.97bp)  ..
            (42);
        \draw [] (51) ..controls (91.896bp,112.44bp) and (99.224bp,87.387bp)
            .. (19);
        \draw [] (51) ..controls (84.379bp,113.42bp) and (61.688bp,87.57bp)  ..
            (50);
        \draw [] (22) ..controls (23.8bp,191.84bp) and (23.8bp,168.29bp)  ..
            (27);
        \draw [] (22) ..controls (27.414bp,192.62bp) and (42.541bp,166.77bp)
            .. (28);
        \draw [] (22) ..controls (20.186bp,192.62bp) and (5.0588bp,166.77bp)
            .. (35);
        \draw [] (23) ..controls (42.186bp,232.22bp) and (27.059bp,206.37bp)
            .. (22);
        \draw [] (23) ..controls (49.414bp,232.22bp) and (64.541bp,206.37bp)
            .. (31);
        \draw [] (23) ..controls (53.028bp,232.22bp) and (83.282bp,206.37bp)
            .. (32);
        \draw [] (23) ..controls (45.8bp,231.44bp) and (45.8bp,207.89bp)  ..
            (36);
        \draw [] (46) ..controls (64.186bp,271.82bp) and (49.059bp,245.97bp)
            .. (23);
        \draw [] (46) ..controls (76.888bp,272.82bp) and (123.7bp,246.15bp)  ..
            (25);
        \draw [] (46) ..controls (71.414bp,271.82bp) and (86.541bp,245.97bp)
            .. (26);
        \draw [] (46) ..controls (67.8bp,271.04bp) and (67.8bp,247.49bp)  ..
            (49);
        \draw [] (25) ..controls (126.57bp,232.22bp) and (96.318bp,206.37bp)
            .. (32);
        \draw [] (25) ..controls (137.41bp,232.22bp) and (152.54bp,206.37bp)
            .. (30);
        \draw [] (25) ..controls (133.8bp,231.44bp) and (133.8bp,207.89bp)  ..
            (38);
        \draw [] (26) ..controls (86.186bp,232.22bp) and (71.059bp,206.37bp)
            .. (31);
        \draw [] (26) ..controls (98.888bp,233.22bp) and (145.7bp,206.55bp)  ..
            (30);
        \draw [] (26) ..controls (93.414bp,232.22bp) and (108.54bp,206.37bp)
            .. (39);
        \draw [] (27) ..controls (31.028bp,153.02bp) and (61.282bp,127.17bp)
            .. (47);
        \draw [] (27) ..controls (23.8bp,152.24bp) and (23.8bp,128.69bp)  ..
            (40);
        \draw [] (31) ..controls (75.028bp,192.62bp) and (105.28bp,166.77bp)
            .. (48);
        \draw [] (31) ..controls (60.572bp,192.62bp) and (30.318bp,166.77bp)
            .. (27);
        \draw [] (31) ..controls (67.8bp,191.84bp) and (67.8bp,168.29bp)  ..
            (44);
        \draw [] (28) ..controls (49.414bp,153.02bp) and (64.541bp,127.17bp)
            .. (47);
        \draw [] (28) ..controls (45.8bp,152.24bp) and (45.8bp,128.69bp)  ..
            (41);
        \draw [] (32) ..controls (93.414bp,192.62bp) and (108.54bp,166.77bp)
            .. (48);
        \draw [] (32) ..controls (82.572bp,192.62bp) and (52.318bp,166.77bp)
            .. (28);
        \draw [] (32) ..controls (89.8bp,191.84bp) and (89.8bp,168.29bp)  ..
            (45);
        \draw [] (47) ..controls (69.896bp,112.44bp) and (77.224bp,87.387bp)
            .. (29);
        \draw [] (47) ..controls (65.704bp,112.44bp) and (58.376bp,87.387bp)
            .. (50);
        \draw [] (30) ..controls (148.57bp,192.62bp) and (118.32bp,166.77bp)
            .. (48);
        \draw [] (30) ..controls (152.19bp,192.62bp) and (137.06bp,166.77bp)
            .. (43);
        \draw [] (35) ..controls (5.414bp,153.02bp) and (20.541bp,127.17bp)  ..
            (40);
        \draw [] (35) ..controls (9.0281bp,153.02bp) and (39.282bp,127.17bp)
            .. (41);
        \draw [] (36) ..controls (38.572bp,192.62bp) and (8.3176bp,166.77bp)
            .. (35);
        \draw [] (36) ..controls (49.414bp,192.62bp) and (64.541bp,166.77bp)
            .. (44);
        \draw [] (36) ..controls (53.028bp,192.62bp) and (83.282bp,166.77bp)
            .. (45);
        \draw [] (49) ..controls (64.186bp,232.22bp) and (49.059bp,206.37bp)
            .. (36);
        \draw [] (49) ..controls (76.888bp,233.22bp) and (123.7bp,206.55bp)  ..
            (38);
        \draw [] (49) ..controls (75.028bp,232.22bp) and (105.28bp,206.37bp)
            .. (39);
        \draw [] (38) ..controls (126.57bp,192.62bp) and (96.318bp,166.77bp)
            .. (45);
        \draw [] (38) ..controls (133.8bp,191.84bp) and (133.8bp,168.29bp)  ..
            (43);
        \draw [] (39) ..controls (104.57bp,192.62bp) and (74.318bp,166.77bp)
            .. (44);
        \draw [] (39) ..controls (115.41bp,192.62bp) and (130.54bp,166.77bp)
            .. (43);
        \draw [] (40) ..controls (29.221bp,113.42bp) and (51.912bp,87.57bp)  ..
            (50);
        \draw [] (44) ..controls (71.414bp,153.02bp) and (86.541bp,127.17bp)
            .. (51);
        \draw [] (44) ..controls (60.572bp,153.02bp) and (30.318bp,127.17bp)
            .. (40);
        \draw [] (41) ..controls (47.896bp,112.44bp) and (55.224bp,87.387bp)
            .. (50);
        \draw [] (45) ..controls (89.8bp,152.24bp) and (89.8bp,128.69bp)  ..
            (51);
        \draw [] (45) ..controls (82.572bp,153.02bp) and (52.318bp,127.17bp)
            .. (41);
        \draw [] (50) ..controls (58.896bp,72.836bp) and (66.224bp,47.787bp)
            .. (42);
        \draw [] (43) ..controls (126.57bp,153.02bp) and (96.318bp,127.17bp)
            .. (51);
    \end{tikzpicture}
    \caption{The Hasse diagram of the set of non-zero Left dead ends born by
    day 4.}
    \label{fig:hasse4}
\end{figure}
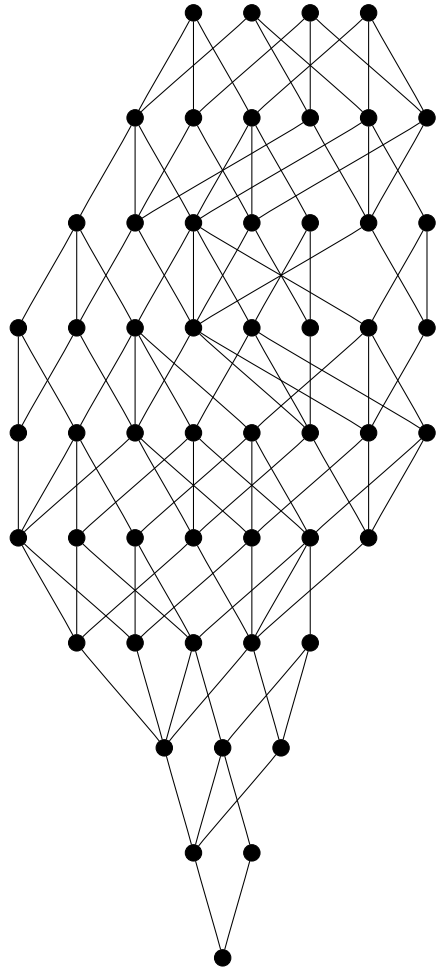

\section{}
\label{app:pocancellativity}

It is not known exactly how pocancellativity and invertibility differ within a
given universe (invertibility, of course, always implies pocancellativity). In
the recent work by the present author, McKay, Milley, Nowakowski, and Santos
\cite{davies.mckay.ea:pocancellation}, however, the universe $\mathcal{E}$ was
shown to be pocancellative (with `$\geq_\mathcal{E}$' as the partial order,
naturally). But $\mathcal{E}$ does not form a group, and so we have an example
of a universe where the submonoid of pocancellative elements (i.e.\ the
pocancellative submonoid) differs from the submonoid of invertible elements
(i.e.\ the invertible submonoid).

The other two universes studied in the paper were $\mathcal{D}$ and
$\mathcal{M}$, which were shown not to be pocancellative. Even though
$\mathcal{M}$ was shown to have a sizable submonoid that is pocancellative (the
set of blocking games $\mathcal{B}$), there were no pocancellative elements
discovered of the full monoid. The situation for $\mathcal{D}$ was more dire:
no pocancellative submonoids were discovered at all, let alone pocancellative
elements.

Given that we have characterisations of the invertible elements of
$\mathcal{D}$ and $\mathcal{E}$, we have an initial question: what do the
pocancellative elements of $\mathcal{D}$ and $\mathcal{E}$ look like, and how
do they differ from the invertible elements of the two monoids respectively? A
grander question, still, is the following.

\begin{problem}
    In a universe $\mathcal{U}$, which elements, if any, are pocancellative but
    not invertible?
\end{problem}

\end{document}